\documentclass[11pt]{amsart}
\usepackage[leqno]{amsmath}
\usepackage{amsthm}
\usepackage{amssymb}
\usepackage{stmaryrd}
\usepackage{mathrsfs}
\usepackage{color}
\usepackage{url}
\usepackage[all]{xy}
\usepackage{supertabular}

\theoremstyle{plain}
\newtheorem{theorem}{Theorem}[subsection]
\newtheorem{lemma}[theorem]{Lemma}
\newtheorem{corollary}[theorem]{Corollary}
\newtheorem{proposition}[theorem]{Proposition}

\theoremstyle{definition}
\newtheorem{definition}[theorem]{Definition}

\theoremstyle{remark}
\newtheorem{remark}[theorem]{Remark}

\def\Z{\mathbf{Z}}
\def\C{\mathbf{C}}
\def\R{\mathbf{R}}
\def\Q{\mathbf{Q}}
\def\F{\mathbf{F}}
\def\P{\mathbf{P}}

\let\ms\mathscr
\let\mf\mathfrak
\let\mc\mathcal
\let\ss\scriptstyle
\let\wt\widetilde
\let\ol\overline

\let\gtimes\ast

\DeclareMathOperator{\Aut}{Aut}

\DeclareMathOperator{\tr}{tr}

\newcommand{\GL}{\mathrm{GL}}
\newcommand{\SL}{\mathrm{SL}}
\newcommand{\PGL}{\mathrm{PGL}}
\newcommand{\PSL}{\mathrm{PSL}}
\newcommand{\spl}{\mathrm{split}}

\def\mat#1#2#3#4{\left( \begin{array}{cc} #1 & #2 \\ #3 & #4 \end{array} \right)}
\def\vec#1#2{\left( \begin{array}{c} #1 \\ #2 \end{array} \right)}

\title{Real components of modular curves}

\author{Andrew Snowden}

\date{August 15, 2011}

\begin{document}

\begin{abstract}
We study the real components of modular curves.  Our main result is an abstract group-theoretic description of
the real components of a modular curve defined by a congruence subgroup of level $N$ in terms of the corresponding
subgroup of $\SL_2(\Z/N\Z)$.  We apply this result to several families of modular curves (such as $X_0(N)$,
$X_1(N)$, etc.) to obtain formulas for the number of real components.  Somewhat surprisingly, the multiplicative
order of 2 modulo $N$ has a strong influence in many cases:  for instance, if $N$ is an odd prime then
the real locus of $X_1(N)$ is connected if and only if $-1$ and 2 generate $(\Z/N\Z)^{\times}$.
\end{abstract}

\maketitle
\tableofcontents

\section{Introduction}

This article is a study of the real components of modular curves.  Our main result is an abstract
group-theoretic description of the real components of a modular curve defined by a congruence subgroup of level $N$
in terms of the corresponding subgroup of $\SL_2(\Z/N\Z)$.  We apply this result to several families of
modular curves to obtain formulas for the number of real components.

\subsection{Description of main result}

We now describe our main result in some detail.  We begin in \S \ref{s:1} and \S \ref{s:2} with a general study of real
components of upper half-plane
quotients.  To this end, let $\Gamma$ be a Fuchsian group and let $X_{\Gamma}=\mf{h}^*/\Gamma$ be the associated
quotient, including cusps.  We assume throughout that $X_{\Gamma}$ is compact.  So far, $X_{\Gamma}$ does not have
any real structure.  To obtain one, we assume that we have an anti-holomorphic involution $c$ of the upper half-plane
(``complex conjugation'') which preserves $\Gamma$.  Thus $c$ descends to $X_{\Gamma}$, and we define $X_{\Gamma}(\R)$
to be its fixed locus.  By the ``real components'' of $X_{\Gamma}$, we mean the connected components of the space
$X_{\Gamma}(\R)$.  Any real component is diffeomorphic to the circle.

Any point on $X_{\Gamma}(\R)$ lifts to a point $z$ in $\mf{h}^*$ satisfying $cz=\gamma z$ for some
$\gamma \in \Gamma$; in fact, this $\gamma$ is ``admissible,'' meaning it satisfies $\gamma^c=\gamma^{-1}$.  For
such $\gamma$, let $C_{\gamma}$ denote the locus in $\mf{h}^*$ defined by the equation $cz=\gamma z$.
If $\Gamma$ acts freely on $\mf{h}$ with compact quotient, it is not difficult to see that mapping $\gamma$ to the
image in $X_{\Gamma}$ of $C_{\gamma}$ defines a bijection between admissible twisted conjugacy classes in $\Gamma$ and
real components of $X_{\Gamma}$.  (We say that $\gamma$ and $\gamma'$ are ``twisted conjugate'' if $\gamma=\delta^c
\gamma' \delta^{-1}$ for some $\delta \in \Gamma$.)  For general $\Gamma$, this result is no longer true:  cusps and
elliptic points interfere.  To remedy this, we define a graph $\Xi_{\Gamma}$ whose vertices are the real cusps and real
elliptic points of even order on $X_{\Gamma}$ and whose edges are the admissible twisted conjugacy classes of
$\Gamma$.  This graph is allowed to have edges which connect to no vertices, which we picture as circles.  We then
show that $\Xi_{\Gamma}$ is homeomorphic to $X_{\Gamma}(\R)$.

In \S \ref{s:3}, we set aside Fuchsian groups and put ourselves in the following abstract situation:  we have a ring
$R$ of finite characteristic (satisfying an additional hypothesis at 2), an involution $C \in \GL_2(R)$ of
determinant $-1$ (``complex conjugation'') and a subgroup $G$ of $\SL_2(R)$ stable under conjugation by $C$ and
containing $-1$.  To this
data, we associate a graph $\Xi_G$.  The vertices of this
graph are of two types:  elliptic and parabolic.  The definition of $\Xi_G$ is somewhat involved in the presence
of elliptic vertices, so for the purposes of the present discussion let us ignore them.  The parabolic vertices of
$\Xi_G$ are the basis vectors in $R^2$ which are eigenvectors of $Cg$ for some admissible $g \in G$, modulo
the action of $G$.  (Admissibility in this setting is defined just as before.)  Two parabolic vertices $x$ and $y$ are
connected if $\langle x, y \rangle$ is $\pm 1$ or $\pm 2$ and there exists $g \in G$ such that
\begin{displaymath}
Cg+1=\left( \frac{2}{\langle x, y \rangle} \right) \langle -, y \rangle x.
\end{displaymath}
The quantity in parentheses is always taken to be $\pm 1$ or $\pm 2$.  Given $x$, $y$ and $g$ as above, one has
$Cgx=x$ and $Cgy=-y$, though the above identity is stronger than these two in general.  We give an explanation for the
pairing condition in \S \ref{ss:role2}.  The main result of \S \ref{s:3} states that the graph $\Xi_G$ is a union of
cycles, i.e., every vertex has valence two and there are no infinitely long paths.

In \S \ref{s:4} we turn to the study of real components for congruence subgroups of $\SL_2(\Z)$.  Thus let $\Gamma$ be
a subgroup of $\SL_2(\Z)$ containing $\Gamma(N)$ for some $N$, and let $c$ be a complex conjugation on $\mf{h}$
preserving $\Gamma$.  We must assume that $c$ also preserves $\Gamma(N)$; without this assumption, the behavior
of the real locus is very different.  Associated to $c$ we have a matrix $C$ in $\GL_2(\R)$, which, in this
setting, we show belongs to $\GL_2(\Z)$.  Let $G$ be the image of $\Gamma$ in $\SL_2(\Z/N\Z)$.  We can then
build two graphs:  $\Xi_{\Gamma}$, constructed in \S \ref{s:2}, and $\Xi_G$, constructed in \S \ref{s:3}.  Our
main theorem is that these two graphs are naturally isomorphic.  This gives a description of $X_{\Gamma}(\R)$
purely in terms of $G$ and $C$.

\subsection{Sample of specific results}
\label{ss:sample}

In \S \ref{s:5} we apply the theory we have developed to several families of modular curves to obtain formulas for
the number of real components.  We state two of those results here.

\begin{proposition}
Let $N$ be a positive integer.  If $N$ is a power of 2 then $X_0(N)$ has one real component.  Otherwise, let $n$ be the
number of distinct odd primes factors of $N$ and let $\epsilon$ be 1 if $N$ is divisible by 8 and 0 otherwise.  Then
$X_0(N)$ has $2^{n+\epsilon-1}$ real components.
\end{proposition}

This result was suggested by Frank Calegari based on computations of William Stein \cite{Stein}.  To state our second
formula, let us introduce some notation.  Let $\phi(N)$ denote the cardinality of $(\Z/N\Z)^{\times}$, as usual.  For
an odd integer $N$, let $\psi(N)$ denote the order of the quotient of the group $(\Z/N\Z)^{\times}$ by the subgroup
generated by $-1$ and 2.  We then have:

\begin{proposition}
Let $N$ be a positive integer, and write $N=2^rN'$ with $N'$ odd.  Then the number of real components of $X_1(N)$ is
given by:
\begin{displaymath}
\begin{cases}
\psi(N') & \textrm{if $r \le 1$} \\
\tfrac{1}{4} \phi(N) & \textrm{if $r \ge 2$ and $N \ne 4$} \\
1 & \textrm{if $N=4$.}
\end{cases}
\end{displaymath}
\end{proposition}

Given the current state of Artin's primitive root conjecture, it seems that one knows that $-1$ and 2 generate
$(\Z/N\Z)^{\times}$ for infinitely many prime numbers $N$ only under the assumption of the generalized Riemann
hypothesis.  Thus one can show that the real locus of $X_1(N)$ is connected for infinitely many prime $N$ only under
this assumption as well.

\subsection{Real component groups of modular Jacobians}

Real components of a curve are closely related to real components of its Jacobian:  if $X/\R$ is a smooth projective
curve with a real point, then the set of real components of its Jacobian is an $\F_2$-vector
space whose dimension is one less than the number of real components of $X$.  See \cite{GrossHarris} for a proof.
Thus our theorems on the real components of modular curves can be translated to theorems about the structure of the
real component group of modular Jacobians.  (In fact, the computations of \cite{Stein} are for Jacobians.)

\subsection{The role of the number 2}
\label{ss:role2}

As we have seen, the number 2 plays a special role in the study of real components of congruence groups.  For
instance, in the abstract setting of \S \ref{s:3} we require that $R$ satisfy special hypotheses at 2, and the formulas
of \S \ref{ss:sample} are very sensitive to the power of 2 in $N$.  Let us now explain how it comes about that 2
is so special. To begin with, the occurrences of 2 just mentioned are easily traced to
the appearance of 2 in the definition of the graph $\Xi_G$:  for two parabolic vertices $x$ and $y$ to be connected
we require $\langle x, y \rangle$ to be $\pm 1$ or $\pm 2$.  To explain the source of this condition, let $\Gamma$
be a subgroup of $\SL_2(\Z)$ without elliptic points and stable under the ``standard'' complex conjugation $c_0$ given
by $x+iy \mapsto x-iy$.  Then an element
\begin{displaymath}
\gamma=\mat{a}{b}{c}{d}
\end{displaymath}
of $\Gamma$ is admissible if and only if $a=d$.  Assuming $\gamma$ is admissible (and $c \ne 0$), the curve
$C_{\gamma}$ contains
the two cusps $\frac{-a \pm 1}{c}$.  We thus have an edge in $\Xi_{\Gamma}$ between these two cusps, and so there is
a corresponding edge in $\Xi_G$.  Recall that parabolic vertices of $\Xi_G$ are represented by elements
of $(\Z/N\Z)^2$.  A cusp $p/q \in \Q\P^1$ of $\Gamma$ corresponds to the parabolic vertex represented by $(p, q)$
in $(\Z/N\Z)^2$, assuming the fraction $p/q$ is in lowest terms.  We therefore see that the edge of $\Xi_G$
corresponding to $\gamma$ connects the two parabolic vertices
\begin{displaymath}
\left( \frac{-a+1}{d_1}, \frac{c}{d_1} \right), \qquad
\left( \frac{-a-1}{d_2}, \frac{c}{d_2} \right),
\end{displaymath}
where $d_1=\gcd(-a+1, c)$ and $d_2=\gcd(-a-1, c)$.  The inner product of these two vectors is
\begin{displaymath}
\frac{2c}{d_1d_2},
\end{displaymath}
which is equal to $\pm 1$ or $\pm 2$.  This fact from elementary number theory is therefore responsible for
the role of 2.

\subsection{Some further results and questions}

We now list some further results we establish:
\begin{itemize}
\item Our proof of the structure theorem for $\Xi_G$ can easily be adapted into an efficient algorithm to
compute real components of modular curves corresponding to real congruence groups.
\item If $\Gamma$ is a real congruence subgroup of odd level without real elliptic points of even order, then this
algorithm simplifies into a formula:  the number of real components of $X_{\Gamma}$ is half the number of orbits
of multiplication by 2 on the set of real cusps.  (Here multiplication by 2 takes a cusp $p/q$, in lowest terms,
to $p'/q'$, where $p'$ and $q'$ are coprime integers equivalent to $2p$ and $2q$ modulo the level.)
\item We establish a product formula for the graph $\Xi$.  Namely, if $G_i$ is a subgroup of $\SL_2(R_i)$ stable under
$C_i$, for $i=1,2$, and $G=G_1 \times G_2$ is the resulting subgroup of $\SL_2(R_1 \times R_2)$, then $\Xi_G$ can be
recovered as a sort of product of $\Xi_{G_1}$ and $\Xi_{G_2}$.  In certain circumstances, this allows one to
reduce computation of $\Xi_G$ to the case of prime power level, which is often easier; see the example of $X_0(N)$
in \S \ref{s:5} for a good illustration of this.
\item In \S \ref{s:2} we classify admissible elements of a real Fuchsian group into several types according to their
twisted centralizers.  In \S \ref{s:5} we give examples showing that every type can occur, even for congruence
subgroup of $\SL_2(\Z)$.
\end{itemize}
Here are some questions we have:
\begin{itemize}
\item We construct the graph $\Xi_G$ in great generality:  $G$ can be any ``real'' subgroup of $\SL_2(R)$, for any
ring $R$.  However, these graphs are only directly related to modular curves when $R=\Z/N\Z$.  Do these graphs have
any meaning for other rings?
\item There are certain real modular curves that do not fit into our theory, such as those corresponding to
non-congruence subgroups, or the twisted form of $X_0(N)$ discussed in \S \ref{s:5}.  Can our theory
be generalized to accommodate these cases?
\item In \cite{Shimura2}, Shimura showed that the canonical model of a Shimura curve which is not a modular curve
does not have any real points.  However, twisted forms of such Shimura curves can have real points.  Can our theory be
extended to cover these curves?
\item Assuming 2 is a primitive root modulo infinitely many primes, there is no bound to the number of cusps that
can occur on real components of $X_1(N)$.  However, no real component of any modular curve we have examined contains
more than 18 elliptic points of even order.  How many elliptic points of even order can occur on a real component?
\end{itemize}

\subsection{Notation and terminology}

We freely use terminology and basic results concerning group actions on the upper half plane. See the first chapter of
\cite{Shimura} for background.  For a ring $R$, we write $\PSL_2(R)$ to mean $\SL_2(R)/\{\pm 1\}$, which is
typically not the set of $R$-points of the variety $\PSL_2$.  We use the term ``graph'' in a very general (even
vague) sense; unless otherwise specified, we allow loops (edges from a vertex to itself), multiple edges between
two vertices and even edges which do not contain any vertices (which we picture as circles).  We write $\phi(N)$
for the order of $(\Z/N\Z)^{\times}$ and $\psi(N)$ for the order of the quotient of this group by the subgroup
generated by $-1$ and 2.

\subsection*{Acknowledgments}

I would like to thank Frank Calegari and William Stein for useful comments.

\section{Real Fuchsian groups}
\label{s:1}

In this section, we introduce the notion of a \emph{real Fuchsian group}, which is a pair consisting of a discrete
subgroup of $\PSL_2(\R)$ and a ``complex conjugation'' of the upper-half plane under which the group stable.

\subsection{Complex conjugations}
\label{ss:cc}

Let $\mf{h}$ denote the upper half-plane.  A \emph{complex conjugation} on $\mf{h}$ is an anti-holomorphic involution
$c:\mf{h} \to \mf{h}$.  For example, the map $c_0:\mf{h} \to \mf{h}$ defined by
\begin{displaymath}
c_0(x+iy)=-x+iy
\end{displaymath}
is a complex conjugation, which we call the \emph{standard complex conjugation}.  It is distinguished in two ways.
First, under the isomorphism of $\mf{h}$ with the unit disc given by the Cayley
transformation, $c_0$ corresponds to the usual complex conjugation on the disc.  Second, a modular form $f$ on
$\mf{h}$ has real Fourier coefficients if and only if $f(c_0 z)=\ol{f(z)}$.  This follows easily from the
identity $\ol{q(z)}=q(c_0z)$, where $q(z)=e^{2\pi i z}$.  Despite the appearance of $c_0$ as preferred, it will
be convenient for us to have the flexibility of allowing arbitrary complex conjugations.

Let $c$ be a complex conjugation on $\mf{h}$.  If $\gamma$ is an element of $\PSL_2(\R)$ then $c \gamma c$ is a
holomorphic automorphism of $\mf{h}$, and therefore given by an element $\gamma^c$ of $\PSL_2(\R)$.  The map
$\gamma \mapsto \gamma^c$ is a group automorphism of $\PSL_2(\R)$ of order 2.  We call an element $\gamma$ of
$\PSL_2(\R)$ \emph{admissible} (with respect to $c$) if $\gamma^c=\gamma^{-1}$.  The map $c\gamma$ of $\mf{h}$ is
a complex conjugation if and only if $\gamma$ is admissible with respect to $c$.  If $c'$ is a second complex
conjugation, then $c'=c\gamma$ for some admissible $\gamma$.

Given an element
\begin{displaymath}
g=\mat{a}{b}{c}{d}
\end{displaymath}
of $\PGL_2(\R)$ of negative determinant and an element $z$ of $\mf{h}$, we put
\begin{displaymath}
gz=\frac{a\ol{z}+b}{c\ol{z}+d},
\end{displaymath}
where $\ol{z}$ denotes the usual complex conjugate of the element $z \in \C\P^1$.  The above formula extends the
action of $\PSL_2(\R)$ on $\mf{h}$ to an action of all of $\PGL_2(\R)$.  Every anti-holomorphic automorphism
of $\mf{h}$ is given by an element of $\PGL_2(\R)$ of negative determinant.

It follows from the above discussion that complex conjugations
on $\mf{h}$ correspond to elements of $\PGL_2(\R)$ of negative determinant which square to the identity.  Given
a complex conjugation $c$, there is a unique matrix $C \in \GL_2(\R)$ (up to signs) of trace 0 and determinant $-1$
which induces $c$ on $\mf{h}$; we call either of $\pm C$ the matrix associated to $c$.  For example, the matrix
associated to $c_0$ is given by
\begin{displaymath}
C_0=\mat{1}{}{}{-1}.
\end{displaymath}
Note that $\gamma$ is admissible with respect to $c$ if and only if the matrix $C \gamma$ has trace 0.
If $C$ and $C'$ are two elements of $\GL_2(\R)$ of trace 0 and determinant $-1$ then there exists
$\gamma \in \PSL_2(\R)$ such that $\gamma C \gamma^{-1}=\pm C'$.  It follows that if $c$ and $c'$ are two
complex conjugations then there exists $\gamma \in \PSL_2(\R)$ such that $c'=\gamma c \gamma^{-1}$.  In particular,
every complex conjugation is conjugate to $c_0$, and therefore has a fixes a point (in fact, an entire line) in
$\mf{h}$.

\subsection{Real Fuchsian groups}

A \emph{real Fuchsian group} is a pair $(\Gamma, c)$ consisting of a discrete subgroup $\Gamma$ of $\PSL_2(\R)$ and
a complex conjugation $c$ on $\mf{h}$ such that $\Gamma$ is stable under the involution $\gamma \mapsto \gamma^c$ of
$\PSL_2(\R)$.  Let $(\Gamma, c)$ be a real Fuchsian group.  Let $\mf{h}^*$ be the space obtained by adding the cusps of
$\Gamma$ to $\mf{h}$ and put $X_{\Gamma}=\mf{h}^*/\Gamma$.  \emph{We assume throughout that $X_{\Gamma}$ is compact.}
We typically regard $X_{\Gamma}$ as just a topological space, though sometimes we will use the additional structure
it has (e.g., smooth manifold).  Since $\Gamma$ is stable
under $c$ the set of cusps of $\Gamma$, and thus $\mf{h}^*$, is stable under $c$.  We thus find that $c$ descends to an
automorphism of $X_{\Gamma}$, which we denote by $c$ and still call complex conjugation.

A \emph{real point} of $X_{\Gamma}$ is a point fixed by $c$.  We write $X_{\Gamma}(\R)$ for the set of real points
of $X_{\Gamma}$.  A simple argument shows that $X_{\Gamma}(\R)$ is diffeomorphic to a \emph{non-empty} disjoint union
of circles; it is non-empty since $c$ fixes a point in $\mf{h}$.  By a \emph{real component} of $X_{\Gamma}$, we mean
one of these circles.

A real Fuchsian group $(\Gamma, c)$ gives a complex orbifold $X_{\Gamma}$ with a finite set of distinguished points
(the cusps) together with an anti-holomorphic involution $c$.  There is an obvious notion of isomorphism for such
pairs $(X_{\Gamma}, c)$, and it is natural to ask how this is reflected in terms of $(\Gamma, c)$.  Clearly, if
$\sigma$ is an admissible element of $\Gamma$ then $(\Gamma, c \sigma)$ gives rise to the \emph{same} quotient
$(X_{\Gamma}, c)$.  Also, it is clear that if $\gamma$ is an element of $\PSL_2(\R)$ then $(\gamma \Gamma \gamma^{-1},
\gamma c \gamma^{-1})$ gives rise to an isomorphic quotient.  In fact, this is all that can happen.  Precisely,
we define two real Fuchsian groups $(\Gamma_1, c_1)$ and $(\Gamma_2, c_2)$ to be \emph{equivalent} if
$(\Gamma_1, c_1)=(\gamma \Gamma_2 \gamma^{-1}, \gamma c_2 \sigma \gamma^{-1})$ for some $\gamma \in \PSL_2(\R)$ and
some $\sigma \in \Gamma_2$ admissible with respect to $c_2$.  One can then show that $(X_{\Gamma_1}, c_1)$ and
$(X_{\Gamma_2}, c_2)$ are isomorphic if and only if $(\Gamma_1, c_1)$ and $(\Gamma_2, c_2)$ are equivalent.

As a final remark, suppose that $\Gamma$ is a discrete subgroup of $\PSL_2(\R)$ such that $X_{\Gamma}$ is compact
and $\ol{c}$ is an anti-holomorphic involution of $X_{\Gamma}$ (respecting the cusps and orbifold structure).  It is
natural to ask if $(X_{\Gamma}, \ol{c})$
comes from a real Fuchsian group.  A necessary condition is that $\ol{c}$ has a fixed point on $X_{\Gamma}$.
In fact, this is sufficient:  if $\ol{c}$ has a fixed point then there is a complex conjugation $c$ on $\mf{h}$
lifting $\ol{c}$ such that $\Gamma$ is stable under the involution $\gamma \mapsto \gamma^c$ of $\PSL_2(\R)$.
Then $(\Gamma, c)$ is a real Fuchsian group giving rise to $(X_{\Gamma}, \ol{c})$.

\subsection{Twisted forms}

We say that two real Fuchsian groups $(\Gamma_1, c_1)$ and $(\Gamma_2, c_2)$ are \emph{twisted forms} of each other
if $\Gamma_1$ and $\Gamma_2$ are conjugate.  This is equivalent to asking for an isomorphism $X_{\Gamma_1} \to
X_{\Gamma_2}$ respecting all the structure, except for complex conjugation.

Suppose that $(\Gamma, c)$ is a real Fuchsian group.  We would like to understand its twisted forms, up to equivalence.
A simple argument shows that every twisted form is equivalent to one of the form $(\Gamma, c\sigma)$, where $\sigma$
is an admissible element of $N\Gamma$, the normalizer of $\Gamma$ in $\PSL_2(\R)$.  Furthermore, two twisted forms
$(\Gamma, c\sigma_1)$ and $(\Gamma, c\sigma_2)$ are equivalent if and only if $\sigma_1$ and $\sigma_2$ are
``twisted conjugate'' in $N \Gamma/\Gamma$, i.e., there exists $\gamma \in N\Gamma$ such that $\gamma^c \sigma_1
\gamma^{-1}$ and $\sigma_2$ map to the same element in $N \Gamma/\Gamma$.

The reader familiar with algebraic geometry may expect that twisted forms of $(X_{\Gamma}, c)$ are classified by
$H^1(G, \Aut(\Gamma))$, where $G$ is the two element group generated by $c$.  This is not quite true, since our
twisted forms all have real points.  In fact, the set of equivalence classes of twisted forms (in our sense) is a
subset of $H^1(G, \Aut(\Gamma))$. The set $H^1(G, \Aut(\Gamma))$ is identified with the set of
twisted conjugacy classes of elements $g \in N\Gamma/\Gamma$ satisfying $g^c=g^{-1}$.  The image of the set of
twisted forms, in our sense, consists of those elements $g$ which lift to an element $\gamma \in N \Gamma$
which still satisfies $\gamma^c=\gamma^{-1}$.

\section{Components for real Fuchsian groups}
\label{s:2}

In this section we define a graph $\Xi_{\Gamma}$ associated to a real Fuchsian group $\Gamma$ and show that it is
homeomorphic to $X_{\Gamma}(\R)$.

\subsection{The graph $\Xi_{\Gamma}$ and the main theorem}
\label{ss:xi1}

We fix for the entirety of \S \ref{s:2} a real Fuchsian group $(\Gamma, c)$.  We let $X_{\Gamma}=\mf{h}^*/\Gamma$
be the corresponding quotient, which we assume, as always, to be compact.  We let $\pi:\mf{h}^* \to X_{\Gamma}$ be the
quotient map.  In this section, we define the graph $\Xi_{\Gamma}$ and state the main theorem of \S \ref{s:2}.

Before defining $\Xi_{\Gamma}$, we must introduce some terminology.  We say that a point of $X_{\Gamma}$ is
\emph{special} if it is a real cusp or a real elliptic point of even order.  For an element $\gamma \in
\Gamma$ we let $C_{\gamma}$ denote the locus in $\mf{h}^*$ consisting of points $z$ which satisfy
$\gamma z=cz$.  We say that two elements $\gamma$ and $\gamma'$ of $\Gamma$ are \emph{twisted conjugate} if there
exists $\delta \in \Gamma$ such that $\gamma=\delta^c \gamma' \delta^{-1}$.  Note that $\delta C_{\gamma}
=C_{\delta^c \gamma \delta^{-1}}$, so if $\gamma$ and $\gamma'$ are twisted conjugate then $\pi(C_{\gamma})
=\pi(C_{\gamma'})$.  Furthermore, note that if $\gamma$ is admissible then so is any twisted conjugate of
$\gamma$.

We now define the graph $\Xi_{\Gamma}$.  The vertex set of $\Xi_{\Gamma}$ is the set of special points of
$X_{\Gamma}$.  The edge set of $\Xi_{\Gamma}$ is the set of admissible twisted conjugacy classes.  A special
point $z$ belongs to the edge corresponding to $\gamma$ if $z$ belongs to $\pi(C_{\gamma})$.  Actually, this definition
must be slightly amended, as follows:  the edge corresponding to $\gamma$ forms a loop at the vertex $z$ if
(and only if) $\pi(C_{\gamma})$ contains an open neighborhood of $z$ in $X_{\Gamma}(\R)$.  We will give a
clearer definition of the graph in \S \ref{s2pf}.  Note that $\Xi_{\Gamma}$ can have loops, multiple edges between
vertices and edges without any vertices.

We can now state the main theorem of \S \ref{s:2}.

\begin{theorem}
\label{s2thm}
The graph $\Xi_{\Gamma}$ is naturally homeomorphic to $X_{\Gamma}(\R)$.
\end{theorem}

In fact, this homeomorphism is the identity map on the vertex set of $\Xi_{\Gamma}$, and maps the interior of the
edge corresponding to $\gamma$ homeomorphically to the interior of the set $\pi(C_{\gamma})$.  This theorem will
take most of \S \ref{s:2} to prove.

\begin{remark}
It is possible to give a definition of $\Xi_{\Gamma}$ which is more group-theoretic than the one we give above,
making less reference to the space $X_{\Gamma}(\R)$.  For instance, the vertex set can be described in terms of
certain subgroups of $\Gamma$ (the stabilizers of special points).  However, we have not found this alternate
definition to be so clear or useful.  For congruence subgroups, however, we will come to a very useful, but very
different, group-theoretic characterization of $\Xi_{\Gamma}$.
\end{remark}

\subsection{The curves $C_{\gamma}$}

We now prove a few simple results about the loci $C_{\gamma}$ introduced above.  We first define some
variants.  Let $\gamma \in \Gamma$.  We define $\ol{C}_{\gamma}$ to be the locus in $\ol{\mf{h}}$ (the union of
$\mf{h}$ and $\R\P^1$) defined by
the equation $\gamma z=cz$.  The locus $C_{\gamma}$ is then $\ol{C}_{\gamma} \cap \mf{h}^*$.  We also put
$C^{\circ}_{\gamma}=\ol{C}_{\gamma} \cap \mf{h}$.  Clearly, any pre-image in $\mf{h}^*$ of a point in
$X_{\Gamma}(\R)$ lies on one of the curves $C_{\gamma}$.

The most basic fact about the $C_{\gamma}$ is the following:

\begin{proposition}
\label{conj:fix}
If $\gamma$ is admissible then $\ol{C}_{\gamma}$ is a semi-circle meeting $\ol{\mf{h}}$ at two distinct points.  If
$\gamma$ is not admissible then $C_{\gamma}$ is empty.
\end{proposition}

\begin{proof}
Let $\gamma \in \Gamma$ be given and write
\begin{displaymath}
C \gamma=\mat{a}{b}{c}{d}.
\end{displaymath}
Suppose $z \in \mf{h}^*$ satisfies $cz=\gamma z$.  Then $(C \gamma)^2=\gamma^c \gamma$ stabilizes $z$, and so
$\vert \tr(\gamma^c \gamma) \vert \le 2$.  A short computation shows that $\gamma^c \gamma$ has trace $2+(a+d)^2$, and
so $a=-d$.  Thus if $C_{\gamma}$ is non-empty then $C \gamma$ has trace 0, and so $\gamma$ is admissible.

Conversely, suppose that $\gamma$ is admissible (i.e., $a=-d$).  It follows from our previous discussion that
$c \gamma$ is a complex conjugation, and thus conjugate to $c_0$, and so fixes a semi-circle.  However, we prefer
to be a bit more explicit.  A short computation shows that $z=x+iy$ belongs to $\ol{C}_{\gamma}$ if and only if
\begin{displaymath}
cx^2+cy^2-2ax+b=0.
\end{displaymath}
When $c=0$, this simply becomes $x=\tfrac{b}{2a}$, which is indeed a semi-circle intersecting $\R\P^1$ at two distinct
points (namely, 0 and $\infty$).  When $c \ne 0$, we can write this equation as
\begin{displaymath}
(x-\tfrac{a}{c})^2+y^2=\tfrac{1}{c^2},
\end{displaymath}
which is the equation of a circle of radius $\tfrac{1}{c}$ centered at the real point $\tfrac{a}{c}$.
It thus meets $\R\P^1$ at two distinct points, namely $\tfrac{a\pm 1}{c}$.
\end{proof}

We note that if $\gamma$ is admissible then $\ol{C}_{\gamma}$ is homeomorphic to a closed interval,
$C^{\circ}_{\gamma}$ to an open interval and $C_{\gamma}$ to an interval, either open, half-open or closed
depending on if the endpoints of $\ol{C}_{\gamma}$ are cusps.  (In fact, we will see that $C_{\gamma}$ is never
half-open.)

The following simple result will be often used, so we state it explicitly:

\begin{proposition}
\label{2-1}
Let $\gamma$ be admissible.  Then any element of $\Gamma$ stabilizing every point of $C_{\gamma}$ is the identity.
\end{proposition}

\begin{proof}
The set of elliptic points of $\Gamma$ is discrete in $\mf{h}$, and so $C_{\gamma}$ contains non-elliptic
points of $\mf{h}$.  Such points, by definition, have trivial stabilizer in $\Gamma$.
\end{proof}

\begin{proposition}
\label{2-2}
If $\gamma$ and $\delta$ are admissible and $C_{\gamma}=C_{\delta}$ then $\gamma=\delta$.
\end{proposition}

\begin{proof}
The transformation $\delta^{-1} \gamma$ stabilizes every element of the semi-circle $C_{\gamma}$, and is thus the
identity.
\end{proof}

\subsection{The space $C_{\gamma}/Z_{\gamma}$}

For $\sigma$ and $\gamma$ in $\Gamma$, we call $\sigma^c \gamma \sigma^{-1}$ the \emph{twisted conjugate} of $\gamma$
by $\sigma$.  Any twisted conjugate of an admissible element is again admissible, and we have $\sigma C_{\gamma}=
C_{\sigma^c \gamma \sigma^{-1}}$.  We let $Z_{\gamma}$ denote the \emph{twisted centralizer} of $\gamma$, i.e., the set
of elements $\sigma \in \Gamma$ such that $\sigma^c \gamma \sigma^{-1}=\gamma$; it is a group.  Fix for the rest of
this section an admissible element $\gamma$ of $\Gamma$.  By Proposition~\ref{2-2}, $Z_{\gamma}$ consists of exactly
those elements $\sigma$ such that $\sigma C_{\gamma}=C_{\gamma}$.  By Proposition~\ref{2-1}, no non-trivial element of
$Z_{\gamma}$ can fix every point of $C_{\gamma}$, and so the map $Z_{\gamma} \to \Aut(C_{\gamma})$ is injective.

\begin{lemma}
\label{lem:z}
The group $Z_{\gamma}$ is either trivial, cyclic of order two, infinite cyclic or infinite dihedral (i.e.,
the semi-direct product $\Z/2\Z \ltimes \Z$).
\end{lemma}

\begin{proof}
These are the only groups which act admit a proper discontinuous actions on the open interval.
\end{proof}

\begin{lemma}
\label{lem:ell}
The even order elliptic points on $C_{\gamma}$ correspond bijectively to the order two elements of $Z_{\gamma}$, with
$z$ corresponding to $\sigma$ if $\sigma$ stabilizes $z$.
\end{lemma}

\begin{proof}
Let $z$ be an even order elliptic point on $C_{\gamma}$ and let $\sigma \in \Gamma_z$ be the unique element of order
two.  One readily verifies that $\tau \mapsto \gamma^{-1} \tau^c \gamma$ defines an automorphism $\Gamma_z \to
\Gamma_z$.  Since $\sigma$ is the unique element of $\Gamma_z$ of order two, it is mapped to itself under this
automorphism.  We thus have $\sigma \in Z_{\gamma}$.

Conversely, let $\sigma \in Z_{\gamma}$ have order two.  Let $x$ and $y$ be the two points of $\ol{C}_{\gamma}$ on
$\R\P^1$.  Since $\sigma$ has finite order, it cannot stabilize elements of $\R\P^1$, and so it must switch
$x$ and $y$.  It follows that $\sigma$ induces an orientation reversing automorphism of $\ol{C}_{\gamma}$.  It
therefore has a unique fixed point $z \in C^{\circ}_{\gamma}$, which, by definition, is elliptic of even order.
\end{proof}

\begin{lemma}
\label{lem:hyper}
If $Z_{\gamma}$ is infinite then $C_{\gamma}$ contains no cusps and is homeomorphic to an open interval.
Any infinite order element of $Z_{\gamma}$ is hyperbolic and induces an orientation preserving map of $C_{\gamma}$.
\end{lemma}

\begin{proof}
Let $\delta$ be an element of $Z_{\gamma}$ of infinite order.  If $\delta$ induced an orientation reversing map of
$C^{\circ}_{\gamma}$ then it would have a fixed point on $C^{\circ}_{\gamma}$, and thus be elliptic, and thus have
finite order. Since this is not the case, $\delta$ must be orientation preserving.  Let $x$ and $y$ be the two
points of $\ol{C}_{\gamma}$ on $\R\P^1$.  As $\delta$ stabilizes each of $x$ and $y$, it is a hyperbolic
transformation.  This shows that $x$ and $y$ cannot be cusps, since the stabilizer of a cusp consists solely of
parabolic elements.
\end{proof}

By a \emph{fundamental domain} for the action of $Z_{\gamma}$ on $C^{\circ}_{\gamma}$ we mean an open subset $\mc{F}
\subset C^{\circ}_{\gamma}$ with the following three properties:  (1) $\mc{F}$ is homeomorphic to an open interval;
(2) no two elements of $\mc{F}$ are equivalent under $Z_{\gamma}$; and (3) every element of $C^{\circ}_{\gamma}$ is
equivalent under $Z_{\gamma}$ to an element of the closure of $\mc{F}$.  One can easily see that a fundamental domain
exists by considering each of the four possibilities for $Z_{\gamma}$ in turn.

\begin{lemma}
\label{lem:inj}
Let $\mc{F} \subset C^{\circ}_{\gamma}$ be a fundamental domain for the action of $Z_{\gamma}$.  Then $\pi$ is
injective on $\mc{F}$.
\end{lemma}

\begin{proof}
We first claim that $\pi$ is a local homeomorphism on $\mc{F}$.  To see this, let $x$ be a point in $\mc{F}$.  We can
then pick an open neighborhood of $x$ homeomorphic to $\R$ and an open neighborhood of $\pi(x)$ homeomorphic to $\R$
such that $\pi$ corresponds to the map $\R \to \R$ given by $x \mapsto x^n$, where $n$ is the order of $\Gamma_x$.  It
follows from Lemma~\ref{lem:ell}, or its proof, that the points to the left and right of an even order elliptic point
on $C_{\gamma}$ are equivalent under $Z_{\gamma}$.  Thus $\mc{F}$ contains no elliptic points of even order.
Since odd powers are local homeomorphisms, the claim follows.

Suppose now that $\pi$ is not injective on $\mc{F}$.  We thus have a local homeomorphism from $\mc{F}$, an open
interval, to a connected component of $X_{\Gamma}(\R)$, a circle, which is not injective.  It follows that
there are disjoint open intervals $U$ and $V$ in $\mc{F}$ such that $\pi(U)=\pi(V)$.  Pick $z \in U$ and $z' \in V$
non-elliptic such that $\pi(z)=\pi(z')$.  Let $\tau$ be an element
of $\Gamma$ such that $\tau z'=z$.  Then $z$ belongs to both $C_{\gamma}$ and $\tau C_{\gamma}=C_{\tau^c \gamma
\tau^{-1}}$, and so $\delta=\gamma^{-1} \tau^c \gamma \tau^{-1}$ stabilizes $z$.  Since $z$ is non-elliptic,
$\delta$ must be the identity element, and so $\tau$ belongs to $Z_{\gamma}$.  This contradicts $\mc{F}$ being a
fundamental domain for $Z_{\gamma}$.  We conclude that $\pi$ is injective on $\mc{F}$.
\end{proof}

\begin{lemma}
\label{lem:cusp}
Let $f:(0,1) \to \mf{h}$ be a continuous map such that as $t$ approaches 1, $f(t)$ converges to an element $x$ in
$\R\P^1$ (for the topology on $\C\P^1$) and $\pi(f(t))$ converges in $X_{\Gamma}$.  Then $x$ is a cusp of $\Gamma$.
\end{lemma}

\begin{proof}
Pick $z \in \mf{h}^*$ such that $\pi(f(t))$ converges to $\pi(z)$ as $t$ approaches 1.  If $z$ is a cusp, let $U$
be the union of $\{z\}$ with an open disc tangent to $\R\P^1$ at $z$ containing no elliptic points and such that
no two of its points are equivalent under $\Gamma$, and put $\Gamma'=1$.
Note that $U$ is open for the topology on $\mf{h}^*$.  If $z$ is elliptic, let $U$ be an open neighborhood of $z$ in
$\mf{h}^*$ stable under $\Gamma'=\Gamma_z$ such that $\pi$ is injective on $U/\Gamma'$ and which contains no elliptic
points other than $z$ and no cusps.  If $z$ is neither elliptic nor cuspidal, let $U$ be an open neighborhood of $z$
containing no elliptic or cuspidal points such that no two of its points are equivalent under $\Gamma$, and put
$\Gamma'=1$.

Now, $\pi(U)$ is an open neighborhood of $\pi(z)$, so for $t$ sufficiently close to 1 we have $\pi(f(t)) \in \pi(U)$.
It follows that for such $t$ we can find $\gamma(t) \in \Gamma/\Gamma'$, necessarily unique, such that $\gamma(t) f(t)$
belongs to $U$.  It is clear that $t \mapsto \gamma(t)$ is locally constant, and therefore constant.  Let $\gamma
\in \Gamma$ be such that $\gamma(t)=\gamma$ for all $t$ close to 1.  Since $\gamma f(t)$ belongs to $U$ for all
$t$ close to 1, it follows that $z$ must be a cusp.  Furthermore, $\gamma f(t)$ must converge to $z$ as $t$
approaches 1, since the closure of $U$ in $\C\P^1$ only intersects $\R\P^1$ at $z$.  It follows that $f(t)$ converges
to $\gamma^{-1} z$ at $t$ approaches 1, and so $x=\gamma^{-1} z$ is a cusp.
\end{proof}

\begin{lemma}
\label{lem:cusp2}
If $Z_{\gamma}$ is finite then $C_{\gamma}$ contains two cusps and is homeomorphic to a closed interval.
\end{lemma}

\begin{proof}
Assume first that $Z_{\gamma}$ is trivial.  Then $C^{\circ}_{\gamma}$ is a fundamental domain for the action of
$Z_{\gamma}$ on $C^{\circ}_{\gamma}$, and so $\pi$ is injective on $C^{\circ}_{\gamma}$.  Let $f:(0,1) \to
C^{\circ}_{\gamma}$ be a homeomorphism.  Since $\pi \circ f$
is a continuous injection from an interval to a circle, the limit of $\pi(f(t))$ as $t$ tends to 1 exists.  Of course,
as $t$ approaches 1, $f(t)$ converges to an element $x$ of $\R\P^1$.  Lemma~\ref{lem:cusp} shows that $x$ is a cusp.
Since $C_{\gamma}$ is closed in $\mf{h}^*$, it contains $x$.  Looking at the behavior near $t=0$, we find another
cusp on $C_{\gamma}$.

The case where $Z_{\gamma}$ has order 2 is similar.  Let $\sigma$ be the non-trivial element of $Z_{\gamma}$.  By
Lemma~\ref{lem:ell}, $\sigma$ fixes a unique element $z$ on $C^{\circ}_{\gamma}$.  Let $C_1$ and $C_2$ be the two
connected components of $C^{\circ}_{\gamma} \setminus \{z\}$.  Both $C_1$ and $C_2$ are fundamental domains for the
action of $Z_{\gamma}$ on $C^{\circ}_{\gamma}$.  Choose a homeomorphism $f:(0,1) \to C_1$ such that $f(t)$ converges to
$z$ as $t$ approaches 0.  Arguing as in the previous paragraph, the limit of $f(t)$ as $t$ approaches 1 is then a cusp
on $C_{\gamma}$.  Looking at $C_2$, we find a second cusp on $C_{\gamma}$.
\end{proof}

Recall that a point of $X_{\Gamma}$ is \emph{special} if it belongs to $X_{\Gamma}(\R)$ and is a cusp or an elliptic
point of even order; we extend this terminology to points of $\mf{h}^*$ as well.

\begin{lemma}
\label{prop:fund}
Let $\mc{F}$ be a fundamental domain for the action of $Z_{\gamma}$ on $C^{\circ}_{\gamma}$.  Then $\mc{F}$ contains
no special points.  Let $\ol{\mc{F}}$ be the closure of $\mc{F}$ in $C_{\gamma}$.  Then
$\ol{\mc{F}}$ is a closed interval and exactly one of the following is true:  the two boundary points of $\ol{\mc{F}}$
are equivalent under $Z_{\gamma}$; or, each boundary point of $\ol{\mc{F}}$ is special.
\end{lemma}

\begin{proof}
We proceed in cases.  First say that $Z_{\gamma}$ is trivial.  Then $C_{\gamma}$ is a closed interval whose endpoints
are cusps (by Lemma~\ref{lem:cusp2}), while $C^{\circ}_{\gamma}$ contains no cusps (obvious) or elliptic points (by
Lemma~\ref{lem:ell}).  Since $\mc{F}=C^{\circ}_{\gamma}$ and $\ol{\mc{F}}=C_{\gamma}$, the proposition follows in this
case.

Next, say that $Z_{\gamma}$ is cyclic of order two.  Then $C_{\gamma}$ is a closed interval whose endpoints are
cusps (by Lemma~\ref{lem:cusp2}).  The curve $C^{\circ}_{\gamma}$ contains a unique even order elliptic point
$z$ (by Lemma~\ref{lem:ell}).  The space $C^{\circ}_{\gamma} \setminus \{z\}$ contains two connected components, and
$\mc{F}$ is equal to one of them.  Clearly, $\mc{F}$ contains no cusp or even order elliptic point.  The closure
$\ol{\mc{F}}$ is the closed interval between $z$ and one of the two cusps on $C_{\gamma}$.  The two boundary points
are obviously inequivalent under $Z_{\gamma}$.

Now say that $Z_{\gamma}$ is infinite cyclic.  We can then find a generator $\delta$ of $Z_{\gamma}$ and a point
$z \in C^{\circ}_{\gamma}$ such that $\mc{F}$ is the open interval between $z$ and $\delta z$.  It follows that
$\ol{\mc{F}}$ is the closed interval between $z$ and $\delta z$, and so its two endpoints belong to the same orbit of
$Z_{\gamma}$.  The interval $\mc{F}$ contains no elliptic points (by Lemma~\ref{lem:ell}) and no cusps (obvious).  For
the same reasons, the boundary points of $Z_{\gamma}$ are not elliptic points or cusps.

Finally, say that $Z_{\gamma}$ is infinite dihedral.
The order two elements of $Z_{\gamma}$ fall into two conjugacy classes.  It follows from Lemma~\ref{lem:ell} that
there are two $Z_{\gamma}$-orbits of even order elliptic elements on $C^{\circ}_{\gamma}$.  The set $\mc{F}$ is the
open interval between two consecutive even order elliptic elements $x$ and $y$; thus $\mc{F}$ contains no even order
elliptic points (and of course it contains no cusps).  The closure $\ol{\mc{F}}$ is the closed interval between $x$
and $y$.  Necessarily, $x$ and $y$ belong to different orbits under $Z_{\gamma}$.
\end{proof}

\begin{proposition}
\label{prop:quo}
The quotient $C_{\gamma}/Z_{\gamma}$ is compact, and therefore either a closed interval or a circle.  A point on the
boundary of $C_{\gamma}/Z_{\gamma}$ is special, while the interior of $C_{\gamma}/Z_{\gamma}$ contains no special
points.  The map $\pi$ is injective on the interior of $C_{\gamma}/Z_{\gamma}$.
\end{proposition}

\begin{proof}
This follows easily from the preceding lemma and Lemma~\ref{lem:inj}.
\end{proof}

\subsection{Classification of admissible elements}
\label{ss:class}

It is useful to classify the admissible elements of $\Gamma$ into four
types, according to their twisted centralizers.  To this end, let $\gamma$ be an admissible element.
\begin{itemize}
\item We say that $\gamma$ is \emph{Type~1} if $Z_{\gamma}$ is trivial.  In this case, $C_{\gamma}/Z_{\gamma}$ is
a closed interval whose boundary points are cusps.  The image $\pi(C_{\gamma})$ is a closed interval if the two cusps
on $C_{\gamma}$ are inequivalent under $\Gamma$, and a circle otherwise.  (We say that $\gamma$ is \emph{Type~1a}
in the first case and \emph{Type~1b} in the second.)
\item We say that $\gamma$ is \emph{Type~2} if $Z_{\gamma}$ is cyclic of order 2.  In this case,
$C_{\gamma}/Z_{\gamma}$ is a closed interval, with one endpoint a cusp and the other an even order elliptic point.
The image $\pi(C_{\gamma})$ is a closed interval.
\item We say that $\gamma$ is \emph{Type~3} if $Z_{\gamma}$ is infinite cyclic.  In this case,
$C_{\gamma}/Z_{\gamma}=\pi(C_{\gamma})$ is a circle containing no special points.
\item We say that $\gamma$ is \emph{Type~4} if $Z_{\gamma}$ is infinite dihedral.  In this case,
$C_{\gamma}/Z_{\gamma}$ is a closed interval whose two endpoints are
even order elliptic elements.  The image $\pi(C_{\gamma})$ is a closed interval if these two elliptic elements are
inequivalent under $\Gamma$, and a circle otherwise.    (We say that $\gamma$ is \emph{Type~4a}
in the first case and \emph{Type~4b} in the second.)
\end{itemize}
If $\gamma$ and $\gamma'$ are twisted conjugate then they are of the same type.  We extend the terminology
of types to the sets $C_{\gamma}$ or $\pi(C_{\gamma})$, e.g., we say $\pi(C_{\gamma})$ is Type~1 if $\gamma$ is.
In \S \ref{s:5}, we will give examples which show that all the above behaviors actually occur.

\subsection{Proof of Theorem~\ref{s2thm}}
\label{s2pf}

We now prove Theorem~\ref{s2thm}.  We first give a clearer definition of the vertex-edge relationships in
$\Xi_{\Gamma}$.  Let $\gamma$ be an admissible element of $\Gamma$.  As we have seen, $C_{\gamma}/Z_{\gamma}$ is
either a circle or a closed interval.  In the first case, the edge of $\Xi_{\Gamma}$ corresponding to $\gamma$
contains no vertices.  In the second case, the two boundary points of $C_{\gamma}/Z_{\gamma}$ are special points,
and their images in $X_{\Gamma}$ are the two special points on the edge corresponding to $\gamma$ in $\Xi_{\Gamma}$.
It is perfectly possible that these two special points of $X_{\Gamma}$ coincide, and in this case the edge
corresponding to $\gamma$ forms a loop.

Given the above description of $\Xi_{\Gamma}$, Theorem~\ref{s2thm} is an immediate consequence of the following
result.

\begin{proposition}
Mapping $\gamma$ to the image in $X_{\Gamma}$ of the interior of $C_{\gamma}/Z_{\gamma}$ induces a bijection
\begin{displaymath}
\{ \textrm{admissible twisted conjugacy classes of $\Gamma$} \} \to
\pi_0(X_{\Gamma}(\R) \setminus Z),
\end{displaymath}
where $Z \subset X_{\Gamma}(\R)$ is the set of special points.
\end{proposition}

\begin{proof}
Let $\gamma$ be an admissible element of $\Gamma$.  Then $C_{\gamma}/Z_{\gamma}$ is either a circle or a closed
interval whose endpoints are special points.  The interior of $C_{\gamma}/Z_{\gamma}$ contains no special points,
so it maps into $X_{\gamma}(\R) \setminus Z$.  If $C_{\gamma}/Z_{\gamma}$ is a circle then $\pi(C_{\gamma})$ is
necessarily a component of $X_{\Gamma}(\R) \setminus Z$.  If $C_{\gamma}/Z_{\gamma}$ is a closed interval, then
its endpoints map into $Z$, and so the image of its interior is a connected component of $X_{\Gamma}(\R) \setminus Z$.
Thus, in all cases, the image of the interior of $C_{\gamma}/Z_{\gamma}$ is mapped to a full connected component of
$X_{\Gamma}(R) \setminus Z$.

If $\gamma'=\sigma^c \gamma \sigma^{-1}$ is a twisted conjugate of $\gamma$, then $C_{\gamma'}=\sigma C_{\gamma}$ and
so the interiors of $C_{\gamma}/Z_{\gamma}$ and $C_{\gamma'}/Z_{\gamma'}$
have the same image in $X_{\Gamma}(\R) \setminus Z$.  This shows that mapping $\gamma$ to the image of the interior
of $C_{\gamma}/Z_{\gamma}$ induces a well-defined map
\begin{displaymath}
\Phi : \{ \textrm{admissible twisted conjugacy classes of $\Gamma$} \} \to
\pi_0(X_{\Gamma}(\R) \setminus Z).
\end{displaymath}
We now show that $\Phi$ is a bijection.

We first show that $\Phi$ is surjective.  Let $x$ be an element of $X_{\Gamma}(\R) \setminus Z$.  Let $z$ be a lift
of $x$ to $\mf{h}^*$.  Then $z$ belongs to $C_{\gamma}$ for some $\gamma$.  Since $z$ is not special, it maps to
the interior of $C_{\gamma}/Z_{\gamma}$.  Thus $x$ belongs to the image of the interior of $C_{\gamma}/Z_{\gamma}$,
and so $\Phi$ is surjective.

We now that $\Phi$ is injective.  Let $\gamma$ and $\delta$ be admissible elements of $\Gamma$ and suppose that the
interiors of $C_{\gamma}/Z_{\gamma}$ and $C_{\delta}/Z_{\delta}$ map to the same component of $X_{\Gamma}(\R)
\setminus Z$.  Let $x$ be an element of this component which is not an odd order elliptic element (and thus not
elliptic or cuspidal).  Lift $x$ to an element $z$ of $C_{\gamma}$ and
$z'$ of $C_{\delta}$.  Pick $\tau$ in $\Gamma$ so that $z=\tau z'$.  Then $z$ belongs to both $C_{\gamma}$ and
$\tau C_{\delta}=C_{\tau^c \delta \tau^{-1}}$, and so $\gamma^{-1} \tau^c \delta \tau^{-1}$ stabilizes $z$.  Since
$z$ is neither elliptic nor cuspidal it has trivial stabilizer, and so $\gamma^{-1} \tau^c \delta \tau^{-1}=1$.  This
shows that $\gamma$ is a twisted conjugate of $\delta$, and so $\Phi$ is injective.  This completes the proof.
\end{proof}

We note the following corollary of the theorem.

\begin{corollary}
\label{cor:count}
The number of special points is equal to the number of admissible twisted conjugacy classes of type 1, 2 and 4.
\end{corollary}

\subsection{Local behavior at a special point}

Let $x$ be a special point in $\mf{h}^*$.  Let $S_x$ denote the set of curves $C_{\gamma}$ which
contain $x$.  Given such a curve $C_{\gamma}$, a small neighborhood of $x$ in $C_{\gamma}$ maps under $\pi$ to
one side of $\pi(x)$.  We can thus define an equivalence relation $\sim$ on $S_x$ by declaring $C_{\gamma}$ and
$C_{\delta}$ equivalent if small neighborhoods of $x$ in them map to the same side of $\pi(x)$.  There are
clearly two equivalence classes.  The group $\Gamma_x$ acts on $S_x$, the element $\sigma \in \Gamma_x$ taking
$C_{\gamma}$ to $\sigma C_{\gamma}$.  We have the following result:

\begin{proposition}
The equivalence classes for $\sim$ on $S_x$ are exactly the orbits of $\Gamma_x$.
\end{proposition}

\begin{proof}
It is clear that
any two elements in an orbit of $\Gamma_x$ belong to the same equivalence class.  We now prove the other direction.
Thus let $C_{\gamma}$ and $C_{\delta}$ be two elements of $S_x$ which are equivalent under $\sim$.  We can then find
non-elliptic points $z \in C^{\circ}_{\gamma}$ and $z' \in C^{\circ}_{\delta}$ which are in small neighborhoods of $x$
and equivalent under $\Gamma$.  Pick $\tau$ so that $\tau z'=z$.  Then $\tau  C_{\delta}=C_{\gamma}$, and so $\tau$
induces a homeomorphism $C_{\delta}/Z_{\delta} \to C_{\gamma}/Z_{\gamma}$.  Each of these spaces is a closed interval,
with one endpoint being represented by $x$.  Since $z$ and $z'$ are both close to $x$, and $\tau$ maps $z'$ to $z$,
it follows that $\tau$ takes $x \in C_{\delta}/Z_{\delta}$ to $x \in C_{\gamma}/Z_{\gamma}$.  We thus find that
$\tau(x)$ belongs to $Z_{\gamma} x$.  Therefore, by replacing $\tau$ with an element of $Z_{\gamma} \tau$, we find
$\tau C_{\delta}=C_{\gamma}$ and $\tau x=x$.  This shows that $C_{\gamma}$ and $C_{\delta}$ are equivalent under
$\Gamma_x$.
\end{proof}

\begin{lemma}
Let $z$ be a point on $C_{\gamma}$ and let $\sigma$ be any element of $\Gamma_z$.  Then $\gamma^{-1} \sigma^c
\gamma=\sigma^{-1}$.
\end{lemma}

\begin{proof}
Clearly, $z$ belongs to $C_{\gamma \sigma}$, and so $\gamma \sigma$ is admissible.  We therefore have
\begin{displaymath}
\gamma^{-1} \sigma^c=(\gamma \sigma)^c=(\gamma \sigma)^{-1}=\sigma^{-1} \gamma^{-1},
\end{displaymath}
which proves the lemma.
\end{proof}

\begin{proposition}
Let $C_{\gamma}$ contain $x$ and let $\sigma$ belong to $\Gamma_x$.  Then $C_{\gamma} \sim C_{\gamma \sigma}$ if and
only if $\sigma$ belongs to $2 \Gamma_x$.
\end{proposition}

\begin{proof}
For $\tau \in \Gamma_x$, we have $\tau C_{\gamma}=C_{\tau^c \gamma \tau^{-1}}=C_{\gamma \tau^{-2}}$.  Thus
$C_{\gamma \sigma}=\tau C_{\gamma}$ if and only if $\sigma=\tau^{-2}$.  It follows that $C_{\gamma \sigma} \sim
C_{\gamma}$ if and only if $\sigma$ is the square of an element of $\Gamma_x$.
\end{proof}

\begin{corollary}
Let $C_{\gamma}$ contain $x$, and let $\sigma$ be a generator of $\Gamma_x$.  Then $C_{\gamma}$ and $C_{\gamma \sigma}$
are inequivalent under $\sim$.
\end{corollary}

This corollary is useful when computing the graph $\Xi$, for if one has found an edge $\pi(C_{\gamma})$ containing
the special point $x$ then the other edge containing $x$ is given by $\pi(C_{\gamma \sigma})$ where $\sigma$ generates
$\Gamma_x$.  Note that it is possible for $\pi(C_{\gamma})$ and $\pi(C_{\gamma \sigma})$ to coincide; when this
happens, $\pi(C_{\gamma})$ forms a loop at $x$.

\section{The graph associated to a real subgroup of $\SL_2(R)$}
\label{s:3}

In this section, we associate a graph to a ``real'' subgroup of $\SL_2(R)$, where $R$ is a finite characteristic ring,
with some additional hypotheses at 2.  The main result we prove about this graph is that it is a union of cycles.
We also prove a result describing how the construction behaves under direct product and inverse image.

\subsection{The graph $\Xi_G$}
\label{ss:xi}

Let $R$ be a ring, let $U$ be a free rank two
$R$-module with a non-degenerate symplectic pairing $\langle, \rangle$ and let $C$ be an $R$-linear involution of $U$
of determinant $-1$.  We write $g \mapsto g^c$ for the involution of $\SL(U)$ induced by conjugation by by $C$.  Let
$G$ be a subgroup of $\SL(U)$ containing $-1$ and stable under $c$ (this is what we mean by a real subgroup).
We aim to define a graph $\Xi_G$ associated to this data.  This graph will depend on $C$, despite its absence from
the notation.

We say that an element of $U$ is a \emph{basis vector} if the $R$-submodule it generates is a summand.  We call
an element $g$ of $G$ \emph{admissible} if $g^c=g^{-1}$.  Note that $g$ is admissible if and only if $(Cg)^2=1$.
Let $\wt{V}_p$ denote the subset of $U/\{\pm 1\}$
consisting of elements which are represented by some basis vector $x$ satisfying $Cgx=x$ for some admissible
$g \in G$.  We call elements of $\wt{V}_p$ (and often the elements of $U$ representing them) \emph{parabolic
vertices}.  We represent parabolic vertices graphically with a solid dot.

Let $\ms{T}$ denote the set of triples $[x, y; z]$ of basis vectors of $U$ satisfying the following two conditions:
\begin{enumerate}
\item We have $\langle x, z \rangle=\langle z, y \rangle=1$.
\item We have $x+y=wz$ for some $w \in \{1, 2\}$.
\end{enumerate}
We refer to $w$ as the \emph{weight} of the triple $[x, y; z]$.  Note that $w=\langle x, y \rangle$.  We define
the \emph{complementary weight} $w'$ to be 2 or 1 depending on if $w$ is 1 or 2.  Note that if $w=1$ or if $R$ has
no non-zero 2-torsion then $z$ is uniquely determined from $x$ and $y$.

Given $[x, y; z]$ in $\ms{T}$ put $\rho([x,y;z])=[z,z-w'x; y]$.  One readily verifies that $\rho$ maps $\ms{T}$ to
itself.  The map $\rho$ interchanges the weight and complementary weight, i.e., $w\rho=w'$ and $w'\rho=w$.  A short
computation shows that $\rho$ has order 8; in fact, $\rho^4([x,y;z])=[-x,-y;-z]$.  The natural action of $G$ on
$\ms{T}$ commutes with $\rho$.

For an element $[x,y;z]$ of $\ms{T}$, consider the following condition:
\begin{enumerate}
\setcounter{enumi}{2}
\item There exists an element $g$ of $G$ such that
\begin{displaymath}
Cg-1=w' \langle -, x \rangle y, \qquad
Cg+1=w' \langle -, y \rangle x.
\end{displaymath}
\end{enumerate}
Of course, the element $g$ is uniquely determined by $x$ and $y$.  In fact, (c) holds if and only if the endomorphism
of $U$ defined by $u \mapsto Cu+w'\langle u, x \rangle Cy$ belongs to $G$.  If condition (c) is satisfied then the
element $g$ is necessarily admissible and we have $Cgx=x$ and $Cgy=-y$, showing that $x$ and $y$ are parabolic
vertices.  One readily verifies that $[x,y;z]$ satisfies (c) if and only if $\rho^2([x,y;z])$ does.

A \emph{geodesic} is an element of $\ms{T}/\langle \rho^2 \rangle$ satisfying condition (c).  Thus, a geodesic
is represented by a triple $[x,y;z]$ satisfying conditions (a)--(c), and the triples $[x,y;z]$ and $\rho^2([x,y;z])$
represent the same geodesic.  The weight of a geodesic is well-defined.  We denote geodesics graphically by either a
single or double line, depending on if the weight is one or two.

The map $\rho$ descends to an involution of $\ms{T}/\langle \rho^2 \rangle$.  However, it need not take geodesics to
geodesics.  We say that two geodesics \emph{intersect} if they form an orbit of $\rho$.  By definition, intersecting
geodesics have complementary weights.  We have the following observation:

\begin{lemma}
\label{intersect}
Let $[x, y; z]$ be a geodesic.  Then $\rho([x,y;z])$ is a geodesic if and only if $G$ contains the
map $\sigma$ defined by $\sigma(x)=-y$ and $\sigma(z)=z-w'y$.  (Note:  $\sigma(y)=x$.)
\end{lemma}

\begin{proof}
Let $g$ and $h$ be defined by
\begin{displaymath}
Cg-1=w'\langle -, x \rangle y, \qquad Ch-1=w\langle -, z \rangle (z-w'x).
\end{displaymath}
A short computation shows that $Cg Ch=g^{-1} h=\sigma$.  Since $g$ belongs to $G$, we find that $h$ belongs to $G$ if
and only if $\sigma$ does.
\end{proof}

An \emph{elliptic vertex} is an unordered pair of intersecting geodesics.  We think of the elliptic vertex as
the intersection of the two geodesics, and use corresponding terminology (e.g., we say that the geodesics contain
the elliptic vertex).  By definition, a geodesic contains at most one elliptic point.  We write $\wt{V}_e$ for the
set of elliptic points.  We represented elliptic vertices graphically with a hollow dot.

We now define a graph $\wt{\Xi}$.  The vertex set of $\wt{\Xi}$ is the disjoint union of $\wt{V}_p$ and $\wt{V}_e$.
The edges of $\wt{\Xi}$ come from geodesics, as follows.  If $[x, y; z]$ is a geodesic containing no elliptic points
then it contributes an edge between $x$ and $y$ in $\wt{\Xi}$.  If $[x, y; z]$ is a geodesic containing an elliptic
point $p$, then it contributes an edge between $x$ and $p$, as well as one between $p$ and $y$.  The edges of
$\wt{\Xi}$ are undirected.  We assign each edge weight 1 or 2 according to the weight of the geodesic giving rise to
it.  Note that it is possible
that there is more than one edge between two vertices of $\wt{\Xi}$.  However, $\wt{\Xi}$ contains no loops.

One easily verifies that $G$ acts on $\wt{\Xi}$.  With this in mind, we can make our main definition:

\begin{definition}
The graph $\Xi=\Xi_G$ is the quotient of $\wt{\Xi}$ by $G$.
\end{definition}

For the sake of clarity, let us elaborate on the definition slightly.  The vertex set of $\Xi$ is the
quotient of the vertex set of $\wt{\Xi}$ by $G$.  The action of $G$ on $\wt{\Xi}$ takes parabolic vertices to
parabolic vertices and elliptic vertices to elliptic vertices, so there is a notion of ``elliptic'' and ``parabolic''
for vertices $\Xi$.  The edge set of $\Xi$ is the quotient of the edge set of $\wt{\Xi}$ by the action of $G$.
The action of $G$ on the edges of $\wt{\Xi}$ respects weight, and so the edges of $\Xi$ have a weight.   We note
that each geodesic of $\wt{\Xi}$ maps to a single edge of $\Xi$.

\subsection{Invariance under inverse image}

Let $R \to R_0$ be a surjection of rings.  We assume that 3 and 4 are non-zero in $R_0$, and thus in $R$.  Let $U$ be a
free $R$-module of rank 2 with complex conjugation $C$, let
$U_0=U \otimes_R R_0$ and let $C_0$ be the induced complex conjugation on $U_0$.  Let $G_0$ be a subgroup of $\SL(U_0)$
stable under $C_0$ and let $G$ be its inverse image in $\SL(U)$.  Clearly, $G$ is stable under $C$.  One therefore has
graphs $\Xi=\Xi_G$ and $\Xi_0=\Xi_{G_0}$.  The purpose of this section is to prove the following theorem:

\begin{theorem}
\label{invimg}
Assume that the map $\SL_2(R) \to \SL_2(R_0)$ is surjective.  Then $\Xi$ and $\Xi_0$ are isomorphic.
\end{theorem}

The hypothesis implies that every element of $G_0$ lifts to one in $G$.  We note that the hypothesis is automatic
if the kernel of $R \to R_0$ is nilpotent, since $\SL_2$ is a smooth group scheme.  To prove the theorem we proceed in
a series of lemmas.  For $x \in U$ we write $\ol{x}$ for its image in $U_0$.

\begin{lemma}
\label{invimg-1}
Any basis vector of $U_0$ can be lifted to a basis vector of $U$.
\end{lemma}

\begin{proof}
Let $\ol{u}$ be a basis vector of $U_0$.  Pick $\ol{v}$ in $U_0$ with $\langle \ol{u}, \ol{v} \rangle=1$ and pick
$x$ and $y$ in $U$ with $\langle x, y \rangle=1$.  Then there exists a unique element $\ol{g} \in \SL(U_0)$ such
that $\ol{u}=\ol{g} \ol{x}$ and $\ol{v}=\ol{g} \ol{y}$.  Since $\SL(U) \to \SL(U_0)$ is surjective, we can lift
$\ol{g}$ to an element $g$ of $\SL(U)$.  We can then take $u=gx$.
\end{proof}

\begin{lemma}
\label{invimg-2}
Let $\ol{u}$ and $\ol{v}$ be elements of $U_0$ satisfying $\langle \ol{u}, \ol{v} \rangle=1$ and let $u$ be a lift
of $\ol{u}$ to a basis vector of $U$.  Then there exists a lift $v$ of $\ol{v}$ with $\langle u, v \rangle=1$.
\end{lemma}

\begin{proof}
Let $v_0$ be any lift of $v$ and let $z \in U$ be such that $\langle u, z \rangle=1$.  We have $\langle u, v_0 \rangle
=1+\alpha$, where $\alpha$ belongs to the kernel of $R \to R_0$.  We can take $v=v_0-\alpha z$.
\end{proof}

\begin{lemma}
\label{invimg-3}
Every parabolic vertex of $\wt{\Xi}_0$ lifts to one of $\wt{\Xi}$.
\end{lemma}

\begin{proof}
Let $\ol{x}$ be a parabolic vertex of $\wt{\Xi}_0$ and let $x$ be a lift of $\ol{x}$ to a basis vector of $U$
(possible by Lemma~\ref{invimg-1}).
Let $\ol{g} \in G_0$ be such that $C_0 \ol{g} \ol{x}=\ol{x}$.  Pick $\ol{z} \in U_0$ with $\langle \ol{x}, \ol{z}
\rangle=1$, so that $C_0 \ol{g} \ol{z}=-\ol{z}+\ol{a} \ol{x}$ for some $\ol{a} \in R_0$.  Let $z$ be a lift of $\ol{z}$
satisfying $\langle x, z \rangle=1$ (possible by Lemma~\ref{invimg-2}) and let $a$ be a lift of $\ol{a}$.  Let $g$ be
the endomorphism of $U$ defined by $gx=Cx$ and $gz=C(-z+ax)$.  Then $g$ has determinant 1 and reduces to $\ol{g}$, and
thus belongs to $G$.  Thus $x$ is parabolic, completing the proof.
\end{proof}

\begin{lemma}
\label{invimg-4}
Two parabolic vertices in $U$ are equivalent if and only if their images in $U_0$ are.
\end{lemma}

\begin{proof}
Let $x$ and $y$ be parabolic vertices in $U_0$ and let $\ol{x}=\ol{h} \ol{y}$ for some $\ol{h} \in G_0$.
Let $h \in G$ be a lift of $\ol{h}$.  Replacing $y$ by $hy$, we may
assume that $x$ and $y$ have the same image in $U_0$.  Let $\{x,z\}$ be a basis of $U$ and write $y=ax+bz$.  Note
that $a$ and $b$ map to 1 and 0 in $R_0$.  Since $y$ is a basis vector, $a$ and $b$ generate the unit ideal of $R$,
and so we have an expression $pa-qb=1$ for some $p,q \in R$.  Let $g$ be the endomorphism of $U$ given by the
matrix
\begin{displaymath}
\mat{a}{(1-a)q}{b}{1+p(1-a)}
\end{displaymath}
in the basis $\{x, z\}$.  Then $g$ has determinant 1 and induces the identity map of $U_0$, and thus belongs to $G$.
Since $gx=y$, we see that $x$ and $y$ are equivalent.
\end{proof}

\begin{lemma}
\label{invimg-5}
Every geodesic in $\wt{\Xi}_0$ lifts to one in $\wt{\Xi}$.
\end{lemma}

\begin{proof}
Let $[\ol{x}, \ol{y}; \ol{z}]$ be a geodesic in $\wt{\Xi}_0$ of weight $w$.  Let $x$ and $z$ be basis vectors of $U$
lifting $\ol{x}$ and $\ol{z}$ and satisfying $\langle x, z \rangle=1$ (possible by Lemmas~\ref{invimg-1}
and~\ref{invimg-2}).  Put $y=wz-x$.  The endomorphism
$C(1+w'\langle -, x \rangle y)$ of $U$ has determinant 1 and lifts the endomorphism $C_0(1+w'\langle -, \ol{x} \rangle
\ol{y})$ of $U_0$.  By hypothesis, the latter belongs to $G_0$, and so we find that the former belongs to $G$.  This
shows that $[x, y; z]$ is a geodesic of $\wt{\Xi}$, which completes the proof.
\end{proof}

\begin{lemma}
\label{invimg-6}
Two geodesics in $\wt{\Xi}$ are equivalent if and only if their images in $\wt{\Xi}_0$ are.
\end{lemma}

\begin{proof}
Let $[x_1, y_1; z_1]$ and $[x_2, y_2; z_2]$ be two geodesics of $\wt{\Xi}$ whose images in $\wt{\Xi}_0$ are
equivalent.  After possibly apply some power of $\rho^2$, we have an element $\ol{g} \in G_0$ such that
$[\ol{x}_1, \ol{y}_1; \ol{z}_1]=\ol{g} [\ol{x}_2, \ol{y}_2; \ol{z}_2]$ holds in $\ms{T}_0$.  Letting $g$ be a lift
of $\ol{g}$ to $G$
and replacing $[x_2, y_2; z_2]$ with $g[x_2, y_2; z_2]$, we can assume that $x_1$ and $x_2$ reduce to the same
element of $U$, as do $z_1$ and $z_2$.  Now, let $h$ be the endomorphism of $U$ defined by $h(x_1)=x_2$ and
$h(z_1)=z_2$.  Then $h$ induces the identity map of $U_0$, and thus belongs to $G$.  Since $y_1=w_1z_1-x_1$ and
$y_2=w_2z_2-x_2$ and $w_1=w_2$, we find that $h(y_1)=y_2$.  Thus $h[x_1,y_1;z_1]=[x_2,y_2;z_2]$, which shows
that the two geodesics are equivalent.
\end{proof}

\begin{lemma}
\label{invimg-7}
A geodesic in $\wt{\Xi}$ intersects another geodesic if and only if its image in $\wt{\Xi}_0$ does.
\end{lemma}

\begin{proof}
Let $[x,y;z]$ be a geodesic of $\wt{\Xi}$ of weight $w$.  By Lemma~\ref{intersect}, $[x,y;z]$ intersects another
geodesic if
and only if the map $\sigma$ of $U$ defined by $\sigma(x)=-y$ and $\sigma(z)=z-w'y$ belongs to $G$.  Similarly,
$[\ol{x},\ol{y};\ol{z}]$ intersects another geodesic if and only if the map $\ol{\sigma}$ of $U_0$ defined by
$\ol{\sigma}(\ol{x})=-\ol{y}$ and $\ol{\sigma}(\ol{y})=\ol{z}-w'\ol{y}$ belongs to $G_0$.  As $\sigma$ reduces to
$\ol{\sigma}$, it follows that $\sigma$ belongs to $G$ if and only if $\ol{\sigma}$ belongs to $G_0$.  This
proves the lemma.
\end{proof}

We can now prove the theorem:

\begin{proof}[Proof of Theorem~\ref{invimg}]
By Lemmas~\ref{invimg-3} and~\ref{invimg-4}, the equivalence classes of parabolic vertices of $\wt{\Xi}$ and
$\wt{\Xi}_0$ are in bijection.  By Lemmas~\ref{invimg-5} and~\ref{invimg-6}, the equivalence classes of geodesics
in $\wt{\Xi}$ and $\wt{\Xi}_0$ are in bijection.  From this and Lemma~\ref{invimg-7}, the equivalences classes of
elliptic points in $\wt{\Xi}$ and $\wt{\Xi}_0$ are in bijection.  Since these bijections are obviously compatible
with how edges are constructed, we find that $\Xi$ and $\Xi_0$ are isomorphic.
\end{proof}

\subsection{The main theorem}

A \emph{cycle} is a connected graph on finitely many vertices in which each vertex has valence 2.  We consider a
single vertex with a self-edge (i.e, a loop) to be a cycle of length 1.  Our main result about $\Xi$ is the
following theorem.  For our ultimate applications, this theorem is in fact not logically necessary; nonetheless,
we feel it is worth including.

\begin{theorem}
\label{xithm}
Suppose $R=\Z/2^r\Z \times R_1$ where $r$ is a non-negative integer and $R_1$ is a ring of odd characteristic.  Then
the graph $\Xi$ is a union of cycles.
\end{theorem}

By Theorem~\ref{invimg}, it suffices to prove the theorem when $r \ge 1$.  We thus make this assumption, to
streamline some arguments.  In the following section, we give a shorter proof when $r=0$ that yields a stronger
result.  We let $t_i$ be the element $(2^i, 0)$ of $R$, and put $t=t_{r-1}$.  Thus $t$ is the unique non-zero
2-torsion element of $R$.  We let $p:R \to \Z/2^r\Z$ and $\pi:\wt{\Xi} \to \Xi$ be the projection maps.  We proceed
with several lemmas.

\begin{lemma}
\label{xithm-1}
Every parabolic vertex belongs to a geodesic.
\end{lemma}

\begin{proof}
Let $x \in U$ be a basis vector and let $g \in G$ be an admissible element such that $Cgx=x$.  Let $u \in U$ be
such that $\langle x, u \rangle=1$ and write $Cgu=-u+ax$ for some $a \in R$.  First suppose that $a$ belongs to $2R$,
and write $a=2b$.  Put $y=u-bx$ and $z=x+y$.  Then $[x,y;z]$ is a weight one geodesic.  Now suppose that $a$ belongs
to $1+2R$, and write $a=2b+1$.  Put $y=2u-ax$ and $z=u-bx$.  Then $[x,y;z]$ is a weight two geodesic.
\end{proof}

\begin{lemma}
\label{xithm-2}
Every parabolic vertex of $\Xi$ has valence at least two.
\end{lemma}

\begin{proof}
Let $x$ be a parabolic vertex of $\wt{\Xi}$, and let $[x,y;z]$ be a geodesic to which $x$ belongs (which exists
by Lemma~\ref{xithm-1}).  If $x$ is
equivalent to $y$ and $[x,y;z]$ does not contain an elliptic vertex, then $\pi(x)$ is contained in a loop and thus
has valence at least two.  We may thus assume that either $x$ and $y$ are inequivalent, or that they are equivalent by
an element $\sigma$ satisfying $\sigma(x)=-y$ and $\sigma(y)=x$.  We must produce a geodesic inequivalent to
$[x,y;z]$ which contains $x$.

{\bf Case 1:} $[x,y;z]$ has weight one.  Observe that $[x,y+tx;z+tx]$ is a geodesic.  We can therefore pick $i \ge 0$
minimal so that $[x,y+t_ix;z+t_ix]$ is a geodesic.  Now, if $[x,y+t_ix;z+t_ix]$ is inequivalent to $[x,y;z]$ then
we are done.  Thus assume the two are equivalent, and let $h \in G$ be such that $h[x,y;z]=[x,y+t_ix;z+t_ix]$.
Now, either $hx=\pm x$ and $hy=\pm(y+t_ix)$ or else $hx=\pm(y+t_i x)$ and $hy=x$.  In the latter case, $\sigma$ is
available, and replacing $h$ with $h \sigma$ moves us to the first case.  We may thus assume we are in the first
case.  Replacing $h$ with $-h$, if necessary, we can assume that $hx=x$ and $hy=y+t_ix$.  Let $g \in G$ be such
that $Cgx=x$ and $Cg(y+t_ix)=-(y+t_ix)$.  We have $Cghx=x$ and $Cghy=-y-t_ix$.  If $i>0$ then we have
$Cgh(y+t_{i-1}x)=-(y+t_{i-1}x)$, which shows that $[x,y+t_{i-1}x;z+t_{i-1}x]$ is a geodesic, contradicting the
minimality of $i$.  Thus we have $i=0$.  Put $y'=2y+t_0x$ and $z'=y+ax$, where $2a=1+t_0$.  Then $[x,y';z']$ forms
a geodesic of weight two, and is thus inequivalent to $[x,y;z]$.

{\bf Case 2:} $[x,y;z]$ has weight two.  Observe that $[x,y;z+tx]$ is a geodesic.  We can therefore pick $i \ge 1$
minimal so that $[x,y+t_ix;z+t_{i-1}x]$ is a geodesic (the base case corresponding to $i=r$; note that $t_r=0$).  If
$[x,y+t_ix;z+t_{i-1}x]$ is inequivalent to $[x,y;z]$ we are done.  Thus assume the two are equivalent and let
$h \in G$ be such that $h[x,y;z]=[x,y+t_ix;z+t_{i-1}x]$.  As in the previous case, after possibly replacing $h$
with $\pm h$ or $\pm h \sigma$, we can assume that $hx=x$ and $hy=y+t_ix$.  Of course, $hz=z+t_{i-1}x$.  Let $g \in G$
be such that $Cg-1=\langle -, x \rangle (y+t_ix)$.  Suppose $i \ge 2$.  Then $Cgh-1=\langle -, x \rangle (y+t_{i-1}x)$,
which shows that $[x,y+t_{i-1} x; z+t_{i-2} x]$ is a geodesic, contradicting the minimality of $i$.  Now suppose that
$i=1$.  Let $a \in R$ be such that $2a=1-t_0$, and put $y'=z-ax$.  Then $[x,y';z+y']$ is a weight one geodesic (note
that $Cgh-1=2 \langle -, x \rangle y'$), and thus inequivalent to $[x,y;z]$.
\end{proof}

\begin{lemma}
\label{xithm-3}
Let $[x,y_1;z_1]$, $[x,y_2;z_2]$ and $[x,y_3;z_3]$ be three geodesics in $\wt{\Xi}$.  Then there exists $g \in G$
such that $gx=x$ and $gy_i=y_j$ and $gz_i=z_j$ for some $i \ne j$.
\end{lemma}

\begin{proof}
Let $w_i=\langle x, y_i \rangle$ and let $k$ be the number of indices $i \in \{1,2,3\}$ for which $w_i=2$.  We
permute the $y_i$ and $z_i$ so that the weight one geodesics appear first.  We proceed
in four cases, depending on the value of $k$.  We let $g_i \in G$ be such that $Cg_ix-1=w_i \langle -, x \rangle y_i$.

\textbf{Case 0:}
$k=0$.  With respect to the basis $\{x,y_1\}$ of $U$ we have
\begin{displaymath}
x=\vec{1}{0}, \quad y_1=\vec{0}{1}, \quad y_2=\vec{a}{1}, \quad y_3=\vec{b}{1}
\end{displaymath}
for some $a$ and $b$ in $R$, and
\begin{displaymath}
g_1^{-1}g_2=\mat{1}{-2a}{}{1}, \qquad g_1^{-1} g_3=\mat{1}{-2b}{}{1}.
\end{displaymath}
Note that $z_i=x+y_i$ for each $i$.  We must show that one of the matrices
\begin{equation}
\label{matlist}
\mat{1}{a}{}{1}, \qquad \mat{1}{b}{}{1}, \qquad \mat{1}{a-b}{}{1}
\end{equation}
belongs to $G$.  If $p(a)=0$ then there is an integer $n$ such that $na=\tfrac{1}{2} a$;
the matrix $(g_1^{-1} g_2)^{-n}$ is then the first one in \eqref{matlist}.  A similar argument works if $p(b)=0$.
We may thus assume that both $p(a)$ and $p(b)$ are non-zero.  We now proceed in two cases depending on how $p(a)$
and $p(b)$ compare.

First, suppose that $p(a)$ and $p(b)$ generate the same ideal of $\Z/2^r\Z$.  We can find $n, m \in \Z$ such that
$2na+2mb=a-b$.  To do this, first solve in $R/2^rR=\Z/2^r\Z$ using the fact that $p(a)-p(b)$ belongs to the ideal
generated by $2p(a)$, then solve in $R_1$ using the fact that 2 is invertible, and finally use the Chinese remainder
theorem.  The matrix $(g_1^{-1} g_2)^{-n} (g_1^{-1} g_3)^{-m}$ is then the third in \eqref{matlist}.

Now suppose that $p(a)$ and $p(b)$ generate different ideals of $\Z/2^r \Z$ --- say $p(b)$ belongs to the one generated
by $2p(a)$.  We can then find $n, m \in \Z$ such that $2na+2mb=-b$, again using the Chinese remainder theorem.
The matrix $(g_1^{-1} g_2)^{-n} (g_1^{-1} g_3)^{-m}$ is then the second in
\eqref{matlist}.

\textbf{Case 1:}
$k=1$.  With respect to the basis $\{x,y_1\}$ of $U$ we have
\begin{displaymath}
x=\vec{1}{0}, \quad y_1=\vec{0}{1}, \quad y_2=\vec{a}{1}, \quad y_3=\vec{b}{2}
\end{displaymath}
for some $a$ and $b$ in $R$, and
\begin{displaymath}
g_1^{-1}g_2=\mat{1}{-2a}{}{1}, \qquad g_1^{-1} g_3=\mat{1}{-b}{}{1}.
\end{displaymath}
Note that $z_i=x+y_i$ for $i \ne 3$.  We must show that the matrix
\begin{displaymath}
\mat{1}{a}{}{1}
\end{displaymath}
belongs to $G$.  Since $y_3$ is a basis vector, we find that $p(b)$ is a unit of $\Z/2^r\Z$; in particular,
$p(a)$ is a multiple of $p(b)$.  We can therefore find $n, m \in \Z$ such that $2na+mb=a$, and so
$(g_1^{-1} g_2)^{-n} (g_1^{-1} g_3)^{-m}$ is the required matrix.

\textbf{Case 2:}
$k=2$.  With respect to the basis $\{x,y_1\}$ of $U$ we have
\begin{displaymath}
y_2=\vec{a}{2}, \qquad z_2=\vec{a'}{1}, \quad y_3=\vec{b}{2}, \qquad z_3=\vec{b'}{1}
\end{displaymath}
for some $a$, $a'$, $b$ and $b'$ in $R$ with $2a'=a+1$ and $2b'=b+1$, and
\begin{displaymath}
g_1^{-1}g_2=\mat{1}{-a}{}{1}, \qquad
g_1^{-1}g_3=\mat{1}{-b}{}{1}.
\end{displaymath}
We must show that
\begin{displaymath}
\mat{1}{a'-b'}{}{1}
\end{displaymath}
belongs to $G$.  Since $y_2$ is a basis vector, $p(a)$ is a unit of $\Z/2^r\Z$, and so $p(a'-b')$ is a multiple of
$p(a)$.  We can therefore find $n,m \in \Z$ with $na-mb=a'-b'$.  The matrix $(g_1^{-1} g_2)^{-n} (g_1^{-1} g_3)^m$
then works.

\textbf{Case 3:}
$k=3$.  With respect to the basis $\{x,z_1\}$ of $U$ we have
\begin{displaymath}
x=\vec{1}{0}, \quad y_i=\vec{a_i}{2}, \qquad z_i=\vec{a'_i}{1}
\end{displaymath}
for $a_i$ and $a'_i$ in $R$ satisfying $2a'_i=a_i+1$ (and, of course, $a'_1=0$), and
\begin{displaymath}
Cg_i=\mat{1}{-a_i}{}{-1}, \qquad g_i^{-1} g_j=\mat{1}{a_i-a_j}{}{1}.
\end{displaymath}
It suffices to show that $G$ contains some matrix of the form
\begin{displaymath}
\mat{1}{a'_i-a'_j}{}{1},
\end{displaymath}
with $i \ne j$.  Now, if $c_1$, $c_2$ and $c_3$ are three elements of $\Z/2^r\Z$ which sum to 0 then $c_i$ is
a multiple of $2c_j$ for some $i \ne j$.  It follows that, after possibly relabeling the indices,
$p(a'_1)-p(a'_3)$ is a multiple of $2(p(a'_1)-p(a'_2))=p(a_1)-p(a_2)$.  We can therefore find $n, m \in \Z$ such
that $n(a_1-a_2)+m(a_1-a_3)=a'_1-a'_3$.  The matrix $(g_1^{-1} g_2)^n (g_1^{-1} g_3)^m$ then works.
\end{proof}

\begin{lemma}
\label{xithm-4}
If $[x, y; z]$ and $[x, y_1; z_1]$ are geodesics of $\wt{\Xi}$, and $x$ and $y$ are equivalent under $G$, then either
$[x, y; z]$ contains an elliptic point or $[x, y; z]$ and $[x, y_1; z_1]$ are equivalent.
\end{lemma}

\begin{proof}
Let $h \in G$ be such that $hy=x$.  Then $[x, y; z]$, $h[x,y; z]=[x, hx; hz]$ and $[x, y_1; z_1]$ are three geodesics
containing $x$.  By the Lemma~\ref{xithm-3}, we can find $h' \in G$ fixing $x$ and carrying
one of these three geodesics to another.  Now, if $h'[x,y;z]=[x, y_1;z_1]$ or $h'h[x, y; z]=[x, y_1; z_1]$ then
$[x, y; z]$ and $[x, y_1; z_1]$ are equivalent and we are done.  Thus assume that $h'h[x, y; z]=[x, y; z]$.  We must
show that $[x, y; z]$ contains an elliptic point.

Put $\sigma=h'h$, so that $\sigma(y)=x$.  The identity $\sigma[x,y;z]=[x,y;z]$ only holds in
$\ms{T}/\langle \rho^2 \rangle$.  In $\ms{T}$ itself, we have $\sigma[x,y;z]=[\sigma(x),x;\sigma(z)]$,
and so $\rho^2 \sigma[x,y;z]=[x,x-w\sigma(z);\sigma(z-w'x)]$, where $w=\langle x, y \rangle$.  Clearly, this must equal
$[x,y;z]$ in $\ms{T}$ itself.  We thus have $x-w\sigma(z)=y$, which yields $\sigma(x)=-y$.  We also have
$\sigma(z-w'x)=z$, which yields $\sigma(z)=z-w'y$.  We conclude from Lemma~\ref{intersect} that $[x,y;z]$ contains an
elliptic point.
\end{proof}

We now complete the proof of the theorem.

\begin{proof}[Proof of Theorem~\ref{xithm}]
Let $x$ be a parabolic vertex of $\wt{\Xi}$.  By Lemma~\ref{xithm-4}, if $\pi(x)$ belongs to a loop (i.e., there
is a geodesic $[x, y; z]$ not containing an elliptic point but with $\pi(x)=\pi(y)$) then any two geodesics containing
$x$ are equivalent.  Thus $\pi(x)$ has valence two.  Suppose then that $\pi(x)$ does not belong to a loop.  By
Lemmas~\ref{xithm-2} and~\ref{xithm-3}, there are exactly two equivalence classes of geodesics containing $x$.
Each contributes one to the valence of $\pi(x)$, and so $\pi(x)$ has valence two.  Thus all parabolic vertices of $\Xi$
have valence two.  By definition, each elliptic vertex of $\wt{\Xi}$ has valence four, being contained in two edges of
each weight.  The edges of equal weight are equivalent (by Lemma~\ref{intersect}), and so each elliptic vertex of
$\Xi$ has valence two.  We have thus shown that all vertices of $\Xi$ have valence two.

It remains to show that the components of $\Xi$ are finite.  Let $[x,y;z]$ be a geodesic of $\wt{\Xi}$ and let $U'$ be
the $\Z$-submodule of $U$ generated by $x$, $y$ and $z$.  Let $e=\pi([x,y;z])$ be the edge of $\Xi$ corresponding
to $[x,y;z]$.  We claim that any edge of $\Xi$ neighboring $e$ is represented by a geodesic whose components belong
to $U'$.  Thus let $e'$ be an edge neighboring $e$.  First suppose that $e$ and $e'$ meet at $\pi(x)$.  Examining the
proof of Lemma~\ref{xithm-2}, we see that we can find a geodesic of $\wt{\Xi}$ which is inequivalent to $[x,y;z]$
but whose components belong to $U'$.  This geodesic must map under $\pi$ to $e'$, which proves the claim in this
case.  If $e$ and $e'$ meet at $\pi(y)$ the argument is similar.  Now suppose that $e$ and $e'$ meet at an elliptic
vertex.  Then $e'$ is the image of the geodesic $\rho[x,y;z]$, whose components do indeed belong to $U'$.  This
completes the proof of the claim.  It now follows from induction that any parabolic vertex of $\Xi$ in the same
component as $x$ is represented by an element of $U'$.  Since $U'$ is a finite set, it follows that there are only
finitely many parabolic vertices in the component containing $x$, and therefore that this component is finite (since
elliptic vertices only connect to parabolic vertices).  This completes the proof.
\end{proof}

\subsection{Theorem~\ref{xithm} in odd characteristic}

We now establish the following strengthening of Theorem~\ref{xithm} in odd characteristic:

\begin{proposition}
\label{xithm-odd}
Suppose that $R$ has odd characteristic.  Then every vertex of $\Xi$ belongs to exactly one edge of each weight
(and to no loops).
\end{proposition}

Note that in any geodesic $[x,y;z]$, the element $z$ is uniquely determined from $x$ and $y$.  Therefore, we drop
the third component of geodesics from notation in this section.  If $x$ and $y$ are basis vectors then $[x,y]$ forms
a geodesic if and only if there exists $g \in G$ with $Cgx=x$ and $Cgy=-y$.  The following lemma, combined with
Theorem~\ref{xithm}, establishes the proposition.

\begin{lemma}
Every parabolic vertex of $\Xi$ belongs to at least one edge of each weight.
\end{lemma}

\begin{proof}
Let $x$ be a parabolic vertex of $\wt{\Xi}$ and pick $g \in G$ admissible so that $Cgx=x$.  Let $u \in U$ be such
that $\langle x, u \rangle=1$, and write $Cgu=-u+ax$ with $a \in R$.  Put $y=u-\tfrac{1}{2} ax$.  Then $Cgy=-y$,
and so $[x, y]$ forms a geodesic of $\wt{\Xi}$, of weight 1.  As $Cg (2y)=-(2y)$, we find that $[x, 2y]$ forms
a geodesic of weight 2.  This completes the proof.
\end{proof}

As the proof of Theorem~\ref{xithm} is quite involved, we offer the following lemma which, together with the above
lemma, directly establishes the proposition --- except for the statement about loops.  (One can give a similar direct
argument to deal with loops.)

\begin{lemma}
Every parabolic vertex of $\Xi$ belongs to at most one edge of each weight.
\end{lemma}

\begin{proof}
It suffices to show that if $[x,y]$ and $[x,y']$ are geodesics of $\wt{\Xi}$ of equal weight then there exists
$h \in G$ such that $hx=x$ and $hy'=y$.  Let $g \in G$ be such that $Cgx=x$, $Cgy=-y$ and
similarly define $g'$.  Since $\{x, y\}$ forms a basis of $U$, we can write $y'=y+ax$ for some
$a \in R$.  We have $CgCg'=g^{-1} g'$, and so $g^{-1}g'x=x$ and $g^{-1}g'y'=y'-2ax$.  Let $n$ be an integer which is
equal to $-\tfrac{1}{2}$ in $R$.  Then $(g^{-1}g')^nx=x$ while $(g^{-1}g')^ny'=y$.  This completes the proof.
\end{proof}

\subsection{Modular graphs}
\label{ss:modgr}

It will be convenient in what follows to formalize some notions.
A \emph{modular graph} is an undirected graph which is a disjoint union of cycles, and such that each vertex is
assigned one of two types (elliptic or parabolic) and each edge is assigned a weight in $\{ 1, 2 \}$.  We furthermore
require that every edge contain a parabolic vertex.  We allow multiple edges between vertices.  We say that a modular
graph is \emph{regular} if every vertex belongs to exactly one edge of each weight.
Of course, $\Xi_G$ is a modular graph by Theorem~\ref{xithm}, and regular when $R$ has odd characteristic by
Proposition~\ref{xithm-odd}.

For a modular graph $\Xi$, write $V_p(\Xi)$ for the set of parabolic vertices, $V_e(\Xi)$ for the set of elliptic
vertices and $V(\Xi)$ for $V_p(\Xi) \amalg V_e(\Xi)$.  Write $E(\Xi) \subset V(\Xi) \times V(\Xi)$ for the set of
edges, and let $s$ and $t$ (source and target) be the two maps $E(\Xi) \to V(\Xi)$.  We encode the fact that $\Xi$ is
undirected in an involution $\tau:E(\Xi) \to E(\Xi)$ satisfying $t(\tau(e))=s(e)$.  We picture $e$ and $\tau(e)$
as the same edge.

Let $\Xi_1$ and $\Xi_2$ be two modular graphs.  We define a new graph $\Xi$ as follows.  We define $V_p(\Xi)$ to be
$V_p(\Xi_1) \times V_p(\Xi_2)$ and we define $V_e(\Xi)$ to be $V_e(\Xi_1) \times V_e(\Xi_2)$.  The set $E(\Xi)$ is the
subset of $E(\Xi_1) \times E(\Xi_2)$ consisting of those pairs $(e_1, e_2)$ of equal weight which satisfy one of
the following conditions:
\begin{itemize}
\item The vertices $s(e_1)$ and $s(e_2)$ are of the same type, as are $t(e_1)$ and $t(e_2)$.
\item The vertices $s(e_1)$ and $s(e_2)$ are parabolic, while one of $t(e_1)$ and $t(e_2)$ is parabolic and the
other is elliptic.
\end{itemize}
Let $(e_1, e_2)$ be an element of $E(\Xi)$.  In the first case above, we define $\tau(e_1, e_2)$ to be
$(\tau(e_1), \tau(e_2))$, we define $s(e_1, e_2)$ to be $(s(e_1), s(e_2))$ and we define $t(e_1, e_2)$ to be
$(t(e_1), t(e_2))$.  Suppose now we are in the second case, with $t(e_1)$ is parabolic and $t(e_2)$ elliptic.
We define $\tau(e_1, e_2)$ to be $(\tau(e_1), e_2)$, we define $s(e_1, e_2)$ to be $(s(e_1), s(e_2))$ and we
define define $t(e_1, e_2)$ to be $(t(e_1), s(e_2))$.  Of course, we define the weight of $(e_1, e_2)$ to be
the weight of $e_1$, which agrees with the weight of $e_2$.  We call $\Xi$ the \emph{product} of $\Xi_1$ and $\Xi_2$
and denote it by $\Xi_1 \gtimes \Xi_2$.  The product is commutative (up to isomorphism) and distributes over disjoint
union.

\begin{proposition}
\label{modprod}
Let $\Xi_1$ and $\Xi_2$ be modular graphs, one of which is regular.  Then $\Xi_1 \ast \Xi_2$ is a modular graph.
\end{proposition}

\begin{proof}
Assume that $\Xi_2$ is regular.  Let $x_1$ be a vertex of $\Xi_1$ and $x_2$ a vertex of $\Xi_2$ of the same type
as $x_1$.  Let $e_1$ be an edge of $\Xi_1$ with $s(e_1)=x_1$.  Then there exists a unique edge $e_2$ of $\Xi_2$
with $s(e_2)=x_2$ and having the same weight as $e_1$.  Then $(e_1, e_2)$ is an edge of $\Xi$ with source $(x_1, x_2)$.
Furthermore, if $(e_1, e_2)$ is an edge of $\Xi$ with source $(x_1, x_2)$ then $s(e_1)=x_1$.  This shows that there is
a bijection between edges of $\Xi$ sourced at $(x_1, x_2)$ and edges of $\Xi_1$ sourced at $x_1$.  Since every vertex
of $\Xi_1$ has valence two, it follows that every vertex of $\Xi$ has valence two as well.  (Note that $\Xi$ is
undirected.)  Furthermore, it is clear that if $(x_1, x_2)$ and $(y_1, y_2)$ belong to the same component of $\Xi$
then $x_1$ and $y_1$ belong to the same component of $\Xi_1$ and $x_2$ and $y_2$ belong to the same component of
$\Xi_2$.  Since the components of $\Xi_1$ and $\Xi_2$ are finite, this shows that the components of $\Xi$ are finite
as well.  This completes the proof.
\end{proof}

The modular graph
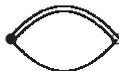
\begin{displaymath}
\begin{xy}
(-7, 0)*{\bullet}; (7, 0)*{\circ};
{\ar@/^2.5ex/@{=}; (-7, 0); (7, 0)};
{\ar@/_2.5ex/@{-}; (-7, 0); (7, 0)};
\end{xy}
\end{displaymath}
is the unique one with one parabolic vertex, one elliptic vertex, one weight one edge and one weight two edge.
It is the identity for the product $\gtimes$.  In fact, $\gtimes$ is best thought of as a sort of fiber product
over the above graph.

\subsection{Behavior of $\Xi$ under direct products}
\label{ss:xiprod}

Let $R$, $U$, etc., be as in the first paragraph of \S \ref{ss:xi}, and let $R'$, $U'$, etc., be defined similarly.
Suppose that $R'$ has odd characteristic.  Let $\Xi=\Xi_G$ and $\Xi'=\Xi_{G'}$, so that $\Xi$ and
$\Xi'$ are modular graphs, with $\Xi'$ regular.  Let $\Xi''=\Xi_{G \times G'}$ be the
graph associated to $G \times G'$, a subgroup of $\SL(U \times U')$.  We then have the following result:

\begin{theorem}
The graph $\Xi''$ is isomorphic to the product $\Xi \gtimes \Xi'$.
\end{theorem}

\begin{proof}
The graph $\Xi''$ is a modular graph by Theorem~\ref{xithm}, while $\Xi \gtimes \Xi'$ is a modular graph by
Proposition~\ref{modprod}.  Furthermore, it is clear that there is a natural bijection between $V_p(\Xi'')$ and
$V_p(\Xi \gtimes \Xi')$.  Thus, to demonstrate the proposition it suffices to show that the following two
statements:  (1) if two parabolic vertices of $\Xi''$ are connected then the corresponding parabolic vertices of
$\Xi \ast \Xi'$ are as well; (2) if two parabolic vertices of $\Xi''$ are connected to a common elliptic vertex then
the corresponding vertices of $\Xi \ast \Xi'$ connect to a common elliptic vertex as well.

We now prove statement (1).  Thus suppose that $x$ and $y$ are parabolic vertices of $\wt{\Xi}$ and $x'$ and $y'$
are parabolic vertices of $\wt{\Xi}'$ such that $(x, x')$ and $(y, y')$ are connected in $\Xi''$.  It follows that,
after replacing these vertices with equivalent ones, we have a geodesic $[(x, x'), (y, y'); (z, z')]$ in $\wt{\Xi}''$
which does not contain an elliptic point.  It is clear then that $[x, y; z]$ and $[x', y'; z']$ are geodesics in
$\wt{\Xi}$ and $\wt{\Xi}'$ of equal weights.  If both contained elliptic points, then a short argument shows that
$[(x, x'), (y, y'); (z, z')]$ would contain an elliptic point; we conclude that one of the two geodesics does not
contain an elliptic point.  Suppose that $[x, y; z]$ does does not contain an elliptic point, so that there is an edge
between $x$ and $y$ in $\Xi$.  If $[x', y'; z']$ also does not contain an elliptic point, then there is an edge between
$x'$ and $y'$ in $\Xi'$, and thus an edge between $(x, x')$ and $(y, y')$ in $\Xi \gtimes \Xi'$.  If $[x', y']$
contains an elliptic point, then we get an edge in $\Xi \gtimes \Xi'$ between $(x, x')$ and $(y, x')=(y, y')$.  Thus,
in all cases, $(x, x')$ and $(y, y')$ are connected in $\Xi \gtimes \Xi'$.

We now prove statement (2).  Thus suppose that $[(x_1, x_1'), (y_1, y_1'); (z_1, z_1')]$ and $[(x_2, x_2'),
(y_2, y_2'); \allowbreak{}(z_2, z_2')]$ are intersecting geodesics in $\wt{\Xi}''$.  Then $[x_1, y_1; z_1]$ and
$[x_2, y_2; z_2]$ are intersecting geodesics in $\wt{\Xi}$, and $[x_1', y_1'; z_1']$ and $[x_2', y_2'; z_2']$ are
intersecting geodesics in $\wt{\Xi}'$.  It follows that $(x_1, x_1')$ and $(x_2, x_2')$ connected to a common elliptic
point in $\Xi \ast \Xi'$, as was to be shown.  This completes the proof.
\end{proof}

\begin{remark}
In many cases, one has a group of the form $\pm(G \times G')$ where $G$ and $G'$ do not contain $-1$.  The above
theorem does not apply to describe the graph associated to this group in terms of $G$ and $G'$.  However, there is
an analogous, though more complicated, result.  In fact, to a group $G$ which does not contain $-1$ one can
associate a graph that is a sort of cover of $\Xi_{\pm G}$, and there is a product theorem for such graphs.
\end{remark}

\section{Components for real congruence groups}
\label{s:4}

The purpose of this section is to describe the real components of $X_{\Gamma}$ when $\Gamma$ is a congruence
subgroup of $\PSL_2(\Z)$ with an appropriate real structure, in terms of the group theory of the corresponding
subgroup of $\SL_2(\Z/N \Z)$.

\subsection{Real congruence groups}

A \emph{real congruence group} is a real Fuchsian group $(\Gamma, c)$ where $\Gamma$ is a subgroup of $\PSL_2(\Z)$ for
which there exists an integer $N \ge 1$ such that $\Gamma$ contains $\Gamma(N)$ and $\Gamma(N)$ is
stable by $c$.  (Recall that $\Gamma(N)$ is the subgroup of $\PSL_2(\Z)$ consisting of matrices which are congruent
to the identity modulo $N$.)

\begin{proposition}
Let $(\Gamma, c)$ be a real congruence group.  Then $c=c_0 \sigma$ for some element $\sigma \in \PSL_2(\Z)$.
\end{proposition}

\begin{proof}
Of course, $c=c_0 \sigma$ for some $\sigma$ in $\PSL_2(\R)$.  Since $\Gamma(N)$ is closed under both
$\gamma \mapsto \gamma^c$ and $\gamma \mapsto \gamma^{c_0}$, it follows that $\sigma$ normalizes $\Gamma(N)$.
As $\PSL_2(\Z)$ is the full normalizer of $\Gamma(N)$ in $\PSL_2(\R)$, we find that $\sigma$ belongs to $\PSL_2(\Z)$,
as was to be shown.
\end{proof}

\begin{corollary}
The matrix $C$ associated to $c$ (see \S \ref{ss:cc}) belongs to $\GL_2(\Z)$.
\end{corollary}

\begin{proof}
Write $c=c_0\sigma$ with $\sigma \in \SL_2(\Z)$, per the proposition.  As we showed, the matrix $C$ is given by $C_0
\sigma$.  Since both $C_0$ and $\sigma$ belong to $\GL_2(\Z)$, the result follows.
\end{proof}

\begin{proposition}
\label{con-cusp}
Let $(\Gamma, c)$ be a real congruence group and let $\gamma \in \Gamma$ be admissible.  Then the curve $C_{\gamma}$
contains two cusps, and is therefore homeomorphic to a closed interval.
\end{proposition}

\begin{proof}
Let $\gamma \in \Gamma$ be admissible with respect to $c$ and write
\begin{displaymath}
C\gamma = \mat{a}{b}{c}{-a},
\end{displaymath}
with $a$, $b$ and $c$ integers.  By Proposition~\ref{conj:fix} (or, rather, its proof), $\ol{C}_{\gamma}$ intersects
$\R\P^1$ at the two points $\tfrac{a \pm 1}{c}$, both of which belong to $\Q\P^1$ and are thus cusps of $\Gamma$.
\end{proof}

\begin{corollary}
\label{cor:cusp}
Every component of $X_{\Gamma}(\R)$ contains a cusp.
\end{corollary}

\begin{corollary}
Every admissible element of $\Gamma$ is Type 1 or Type 2.
\end{corollary}

Combining the above with Corollary~\ref{cor:count}, we obtain the following:

\begin{corollary}
The number of admissible twisted conjugacy classes in $\Gamma$ is equal to the number of special points on $X_{\Gamma}$.
\end{corollary}

Corollary~\ref{cor:cusp} is
extremely useful, because the set of cusps is easy to understand, and so to determine $\Xi_{\Gamma}$ we just need
to determine how the cusps (and elliptic points) are connected.  We remark that the above results do not hold if we
drop the assumption that $\Gamma(N)$ is stable by $c$:  there exist real Fuchsian groups $(\Gamma, c)$ with $\Gamma$
a congruence subgroup of $\PSL_2(\Z)$ which have Type 3 elements (see \S \ref{ss:ex7}).

\subsection{The main theorem}

Let $(\Gamma, c)$ be a real congruence group.  Let $\Xi^{\dag}$ denote the graph associated to $\Gamma$ in
\S \ref{ss:xi1}.  Let $\Xi$ denote the graph associated to $\Gamma$ in \S \ref{ss:xi} (taking $R=\Z$, $U=\Z^2$,
$\langle, \rangle$ the standard symplectic pairing on $U$, $C$ the matrix associated to $c$ and $G=\Gamma$).  The
following is the main result of \S \ref{s:4}.

\begin{theorem}
\label{xi-ident}
Let $(\Gamma, c)$ be a real congruence group.  Then $\Xi^{\dag}=\Xi$.
\end{theorem}

Before proceeding, let us note the following corollary.  Let $G_0$ be the image of $\Gamma$ in $\SL_2(\Z/N\Z)$, let
$U_0=U/NU$ and let $\Xi_0$ be the graph associated to this data in \S \ref{ss:xi}.  By Theorem~\ref{invimg}, we
have $\Xi=\Xi_0$.  (Note:  the surjectivity of the map $\SL_2(\Z) \to \SL_2(\Z/N\Z)$ is proved in
\cite[Lem.~1.38]{Shimura}.)   We therefore obtain the following corollary, which computes the real locus of the curve
$X_{\Gamma}$ in terms of the subgroup of $\SL_2(\Z/N\Z)$ to which $\Gamma$ corresponds.

\begin{corollary}
We have $\Xi^{\dag}=\Xi_0$.
\end{corollary}

We now begin on the proof of the theorem.  Let $U'$ denote the set of basis vectors in $U$.  There is a bijection
\begin{displaymath}
i:\Q\P^1 \to U'/\{\pm 1\},
\end{displaymath}
defined by mapping a point $[p:q]$ in $\Q\P^1$ with $p$ and $q$ coprime to the vector $(p, q) \in U$.  This map is
$\Gamma$-equivariant, and transforms the action of $c$ on the left to that of $C$ on the right.  It follows that the
cusps of $X_{\Gamma}$ (i.e., the orbits of $\Gamma$ on $\Q\P^1$) are identified via this map with $U'/\Gamma$.
Furthermore, a cusp belongs to $X_{\Gamma}(\R)$ if and only if the corresponding element of $U'/\Gamma$ is invariant
under $C$.  This shows that the parabolic vertices of $\Xi^{\dag}$ and $\Xi$ are in natural bijection.

We now prove two lemmas, from which the theorem easily follows.

\begin{lemma}
Let $\gamma \in \Gamma$ be admissible and let $\ol{x}$ and $\ol{y}$ be the two cusps on $C_{\gamma}$.  Then there
exists a unique geodesic $[x, y; z]$ of $\wt{\Xi}$ such that the images of $x$ and $y$ in $U'/\{\pm 1\}$ coincide with
$i(\ol{x})$ and $i(\ol{y})$.  This geodesic is characterized by the identity $C\gamma-1=w' \langle -, x \rangle y$.
\end{lemma}

\begin{proof}
Write
\begin{displaymath}
C \gamma=\mat{a}{b}{c}{-a}.
\end{displaymath}
Suppose $c \ne 0$.  Then, after possibly switching $\ol{x}$ and $\ol{y}$,
\begin{displaymath}
\ol{x}=\frac{a-1}{c}, \qquad \ol{y}=\frac{a+1}{c}.
\end{displaymath}
(See the proof of Proposition~\ref{con-cusp}.)  Let $d_1=\gcd(a-1,c)$ and $d_2=\gcd(a+1,c)$.  Then, by definition,
\begin{displaymath}
i(\ol{x})=\pm \left( \frac{a-1}{d_1}, \frac{c}{d_1} \right), \qquad
i(\ol{y})=\pm \left( \frac{a+1}{d_2}, \frac{c}{d_2} \right),
\end{displaymath}
and so
\begin{displaymath}
\langle i(\ol{x}), i(\ol{y}) \rangle=\pm \frac{2c}{d_1d_2}.
\end{displaymath}
An elementary argument shows that $d_1d_2$ is equal to either $\pm c$ or $\pm 2c$.  We can thus pick lifts $x$ and
$y$ of $i(\ol{x})$ and $i(\ol{y})$ to $U$ so that $w=\langle x, y \rangle$ is 1 or 2.  From the identities
$\gamma \ol{x}=c\ol{x}$ and $\gamma \ol{y}=c\ol{y}$, we conclude that $C \gamma x=\pm x$ and $C \gamma y=\mp y$.
Replacing $\gamma$ with $-\gamma$ if necessary, we can assume that $C\gamma x=x$ and $C\gamma y=-y$, and so $C\gamma-1
=w'\langle -, x \rangle y$.  Let $z$ be the unique element of $U$ satisfying $x+y=wz$.  Then $[x,y;z]$ is a geodesic of
$\wt{\Xi}$.  If $[x',y';z']$ is another geodesic such that $x'$ lifts $i(\ol{x})$ and $y'$ lifts $i(\ol{y})$, then
necessarily $x'=\pm x$ and $y'=\pm y$, and so $[x',y';z']=[x,y;z]$.  If $b \ne 0$ the proof is similar.  When
$b=c=0$ one can argue directly.  This completes the proof.
\end{proof}

\begin{lemma}
Let $\gamma_1 \ne \gamma_2$ be admissible.  Then the curves $C_{\gamma_1}$ and $C_{\gamma_2}$ intersect at an even order
elliptic point if and only if the geodesics of $\wt{\Xi}$ corresponding to $\gamma_1$ and $\gamma_2$ intersect (in the
sense of \S \ref{ss:xi}).
\end{lemma}

\begin{proof}
Let $\ol{x}_i$ and $\ol{y}_i$ be the cusps on $C_{\gamma_i}$ and let $[x_i,y_i;z_i]$ be the corresponding geodesic of
$\wt{\Xi}$.  If $[x_1,y_1;z_1]$ and $[x_2,y_2;z_2]$ intersect, then by Lemma~\ref{intersect} there exists $\sigma \in
\Gamma$ such that $\sigma(\ol{x}_i)=\ol{y}_i$ and $\sigma(\ol{y}_i)=\ol{x}_i$ and $\sigma^2=-1$.  Thus $\sigma$ induces
an orientation-reversing homeomorphism of $C_{\gamma_i}$, and so $C_{\gamma_i}$ contains a fixed point of $\sigma$.
As $\sigma$ has only one fixed point in the upper half-plane, an elliptic point of even order, this point belongs to
$C_{\gamma_1} \cap C_{\gamma_2}$.

Now suppose that $C_{\gamma_1}$ and $C_{\gamma_2}$ intersect at an elliptic point, and let $\sigma \in \Gamma$ generate
the stabilizer of this elliptic point.  Then $\sigma(\ol{x}_i)=\ol{y}_i$ and $\sigma(\ol{y}_i)=\ol{x}_i$ and
$\gamma_2=\gamma_1 \sigma$.  Replacing
$\sigma$ with $-\sigma$ if necessary, we have $\sigma(x_i)=-y_i$ and $\sigma(y_i)=x_i$.  It follows that
$\sigma(z_i)=z_i-w_i' y_i$.  Thus by Lemma~\ref{intersect}, $[x_3,y_3;z_3]=\rho([x_1,y_1;z_1])$ is a geodesic.  We
claim that it coincides with $[x_2,y_2;z_2]$.  To see this, let $\gamma'$ be defined by $C\gamma'-1=w_3' \langle -, x_3
\rangle y_3$.  Then, by the first paragraph, $C_{\gamma'}$ meets $C_{\gamma_1}$ at an even order elliptic point.
Thus $\gamma'=\gamma_1 \sigma$, which shows that $\gamma'=\gamma_2$.  Since $[x_2,y_2;z_2]$ and $[x_3,y_3;z_3]$ are
both associated to $\gamma_2$, they must coincide.
\end{proof}

We now complete the proof of the theorem.

\begin{proof}[Proof of Theorem~\ref{xi-ident}]
We have shown that the parabolic vertices of $\wt{\Xi}$ are in bijection with the real cusps of $\Gamma$ and that
the geodesics of $\wt{\Xi}$ are in bijection with the curves $C_{\gamma}$ with $\gamma$ admissible.  Furthermore,
two geodesics intersect if and only if the corresponding curves intersect at an even order elliptic point.  It now
follows simply from the constructions of $\Xi^{\dag}$ and $\Xi$ that they are isomorphic.
\end{proof}

\subsection{A result on real elliptic points}

The following result rules out even order real elliptic points in many examples.

\begin{proposition}
\label{realell}
Let $(\Gamma, c)$ be a real congruence group with $c=c_0$ and such that $\Gamma$ is contained in $\Gamma_0(N)$ for some
$N>2$.  Then $\Gamma$ has no real elliptic point of even order.
\end{proposition}

\begin{proof}
Let $z \in \mf{h}$ be a real even order elliptic point for $\Gamma$.  Let $\sigma$ be a non-trivial order two element
of $\Gamma$ stabilizing it.  Since every even order elliptic point for $\PSL_2(\Z)$ belongs to the orbit of $i$,
we can write $z=\tau(i)$, for some $\tau \in \PSL_2(\Z)$.  We then have $\sigma=\tau \sigma_0^{-1} \tau^{-1}$,
where
\begin{displaymath}
\sigma_0=\mat{}{1}{-1}{}.
\end{displaymath}
Since $z$ is real, it belongs to $C_{\gamma}$ for some $\gamma \in \Gamma$.  It follows that $i$ is contained on
the curve $\tau^{-1} C_{\gamma}=C_{(\tau^c)^{-1} \gamma \tau}$.  However, one can verify that if $\delta$ is an
element of $\PSL_2(\Z)$ such that $C_{\delta}$ contains $i$, then $\delta$ is either 1 or $\sigma_0$.  We thus find
that $\gamma$ is either $\tau^c \tau^{-1}$ or $\tau^c \sigma_0 \tau^{-1}$; clearly, $\gamma \sigma$ is the other.

We have thus shown that if $\Gamma$ has a real even order elliptic element, then there exists $\tau \in \PSL_2(\Z)$
such that $\tau \sigma_0^{-1} \tau^{-1}$ and $\tau^c \tau^{-1}$ both belong to $\Gamma$.  Writing
\begin{displaymath}
\tau=\mat{a}{b}{c}{d},
\end{displaymath}
we find that the bottom left entry of $\tau \sigma_0^{-1} \tau^{-1}$ is $-(c^2+d^2)$, while the bottom left
entry of $\tau^c \tau^{-1}$ is $-2cd$.  Thus both $c^2+d^2$ and $2cd$ are divisible by $N$.  On the other hand,
$c$ and $d$ are coprime.  This cannot happen unless $N$ is either 1 or 2.
\end{proof}

\section{Examples}
\label{s:5}

We now give some examples.  Throughout this section, $N$ denotes a positive integer.  We write $N=2^r N'$ where
$N'$ is odd.  When $N$ is even, we let $t=2^{r-1} N'$ be the unique non-zero 2-torsion element of $\Z/N\Z$.  We
mostly work with congruence subgroups $\Gamma$ of level $N$.  We always denote by $G$ the corresponding subgroup of
$\SL_2(\Z/N\Z)$.

\subsection{The full modular group}

We begin with the simplest example, namely $\Gamma=\SL_2(\Z)$ and $c=c_0$.  We could apply our theory to this
example, but we find it clearer to reason directly.  The curve $X_{\Gamma}$ has genus 0 and a real point and
is therefore isomorphic to $\P^1$ over $\R$.  It follows that $X_{\Gamma}(\R)$ is a circle.  Put
\begin{displaymath}
\sigma=\mat{}{1}{-1}{}.
\end{displaymath}
Then $C_1$ is the positive imaginary axis, while $C_{\sigma}$ is the unit circle in the upper half-plane.  The
two curves intersect at the elliptic point $i$.  From the standard description of the fundamental domain of
$\Gamma$, we see that the loci $C_1$ and $C_{\sigma}$ are inequivalent, and that $\pi(C_1)$ and $\pi(C_{\sigma})$
intersect only at $\pi(i)$ and $\pi(\infty)$.  It follows that the images of $\pi(C_1)$ and $\pi(C_{\sigma})$ cover
the entire real locus.  The picture is thus:
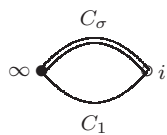
\begin{displaymath}
\begin{xy}
(-10, 0)*{\ss \infty}; (9, 0)*{\ss i};
(0, 7)*{\ss C_{\sigma}}; (0, -7)*{\ss C_1};
(-7, 0)*{\bullet}; (7, 0)*{\circ};
{\ar@/^2.5ex/@{=}; (-7, 0); (7, 0)};
{\ar@/_2.5ex/@{-}; (-7, 0); (7, 0)};
\end{xy}
\end{displaymath}
In this picture, we have indicated the weights of the two edges, which can be computed by taking the pairing
of the two cusps on each curve.  (For $C_1$ the two cusps are 0 and $\infty$, which pair to 1, while for $C_{\sigma}$
they are $\pm 1$, which pair to 2.)

It is instructive to visualize walking around this circle while at the same time moving in the upper half-plane.
Start at $\infty$ in the upper half-plane, which is the parabolic vertex in the above graph.  As we move down the
imaginary axis towards $i$, we approach the elliptic vertex along the weight one edge.  Of course, we reach $i$
when we reach the elliptic vertex.  If we continue to move down the imaginary axis after reaching $i$ then we
start to move backwards along the weight one edge, reaching the parabolic vertex when we hit 0.  To move along the
weight two edge, we must turn at $i$ and move along the unit circle (say clockwise).  If we follow this to the real
axis we reach 1, which is the parabolic vertex.  However, this picture is not completely satisfying since we have
returned to a different cusp.  Instead, one can travel clockwise from $i$ to $e^{2\pi i/6}$, the elliptic point of
order 3, and then turn and move up the line with real part $\tfrac{1}{2}$.  This leaves $C_{\sigma}$ but moves onto
an equivalent curve.  Traveling all the way up, we return to $\infty$.

\subsection{The curve $X^+(N)$}
\label{ss:ex1}

Let $\Gamma=\Gamma^+(N)$ denote the real congruence group $(\Gamma(N), c_0)$ and let $X=X^+(N)$ denote the
corresponding quotient.  We have $G=\{\pm 1\}$.  The space $X$ parameterizes pairs $(E, i)$ where $E$
is an elliptic curve and $i$ is an isomorphism of $E[N]$ with $(\Z/N\Z) \oplus \mu_N$ under which the Weil pairing
corresponds to the standard $\mu_N$-valued pairing on the target.  In what follows, we analyze the structure of the
graph $\Xi=\Xi_{\Gamma}$.  Note that for $N \ne 1$ the group $\Gamma$ has no elliptic elements.  We assume $N>1$
in what follows.

Suppose that $r=0$, i.e., $N$ is odd, and $N \ne 1$.  The real cusps are then represented by the vectors $(a, 0)$ and
$(0, a)$, with $a$ coprime to $N$.  There are thus $\phi(N)$ real cusps.  The local structure of $\Xi$ at the vertex
$(a, 0)$ is given as follows:
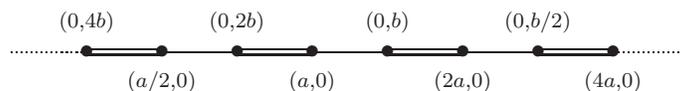
\begin{displaymath}
\begin{xy}
(-35, 4)*{\ss (0, 4b)}; (-25, -4)*{\ss (a/2, 0)}; (-15, 4)*{\ss (0, 2b)}; (-5, -4)*{\ss (a, 0)};
(5, 4)*{\ss (0, b)}; (15, -4)*{\ss (2a, 0)}; (25, 4)*{\ss (0, b/2)}; (35, -4)*{\ss (4a, 0)};
(-35, 0)*{\bullet}; (-25, 0)*{\bullet}; (-15, 0)*{\bullet}; (-5, 0)*{\bullet};
(35, 0)*{\bullet}; (25, 0)*{\bullet}; (15, 0)*{\bullet}; (5, 0)*{\bullet};
{\ar@{..}; (-45, 0); (-35, 0)};
{\ar@{=}; (-35, 0); (-25, 0)};
{\ar@{-}; (-25, 0); (-15, 0)};
{\ar@{=}; (-15, 0); (-5, 0)};
{\ar@{-}; (-5, 0); (5, 0)};
{\ar@{=}; (5, 0); (15, 0)};
{\ar@{-}; (15, 0); (25, 0)};
{\ar@{=}; (25, 0); (35, 0)};
{\ar@{..}; (35, 0); (45, 0)};
\end{xy}
\end{displaymath}
where $ab=1$ in $\Z/N\Z$.  We thus see that every cycle of $\Xi$ contains a vertex of the form $(\ast, 0)$ and that
$(a, 0)$ and $(a', 0)$ belong to the same cycle if and only if $a'$ is of the form $\pm 2^n a$.  It
follows that there are $\psi(N)$ cycles.

Now suppose that $r=1$, i.e., 2 divides $N$ exactly.  If $N=2$ then there are two real cusps in a single cycle.
Thus assume $N \ne 2$, i.e., $N' \ne 1$.  The real cusps are then represented by the
vectors $(a, 0)$ and $(0, a)$, with $a$ coprime to $N$, and the vectors $(a, t)$ and
$(t, a)$, with $a$ coprime to $N'$.  There are $3\phi(N)$ real cusps.  Every cycle of $\Xi$ has six cusps, and
is of the form
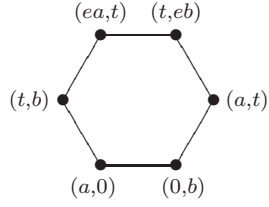
\begin{displaymath}
\begin{xy}
(10, 0)*{\bullet}; (5, 8.66)*{\bullet}; (-5, 8.66)*{\bullet};
(-10, 0)*{\bullet}; (-5, -8.66)*{\bullet}; (5, -8.66)*{\bullet};
(14.5, 0)*{\ss (a, t)}; (5, 11.66)*{\ss (t, eb)}; (-5, 11.66)*{\ss (ea, t)};
(-14.5, 0)*{\ss (t, b)}; (-6, -11.66)*{\ss (a, 0)}; (6, -11.66)*{\ss (0, b)};
{\ar@{-}; (10, 0); (5, 8.66)};
{\ar@{-}; (5, 8.66); (-5, 8.66)};
{\ar@{-}; (-5, 8.66); (-10, 0)};
{\ar@{-}; (-10, 0); (-5, -8.66)};
{\ar@{-}; (-5, -8.66); (5, -8.66)};
{\ar@{-}; (5, -8.66); (10, 0)};
\end{xy}
\end{displaymath}
with $a$ coprime to $N$, $ab=1$ and $e=1-t$.  Note that $e^2=e$ and $t^2=t$.  There are $\tfrac{1}{2} \phi(N)$
cycles.

Finally suppose that $r>1$, i.e., $N$ is divisible by 4.  The real cusps are then represented by the vectors
$(a, 0)$, $(0, a)$, $(a, t)$ and $(t, a)$, with $a$ coprime to $N$.  There are thus $2\phi(N)$ real
cusps.  Every cycle of $\Xi$ has exactly four cusps and is of the form
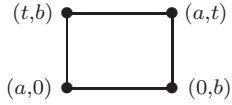
\begin{displaymath}
\begin{xy}
(-12, -5)*{\ss (a, 0)}; (12, -5)*{\ss (0, b)}; (11.5, 5)*{\ss (a, t)}; (-11.5, 5)*{\ss (t, b)};
(-7, -5)*{\bullet}; (-7, 5)*{\bullet}; (7, -5)*{\bullet}; (7, 5)*{\bullet};
{\ar@{-}; (-7, -5); (-7, 5)};
{\ar@{-}; (-7, 5); (7, 5)};
{\ar@{-}; (7, 5); (7, -5)};
{\ar@{-}; (7, -5); (-7, -5)};
\end{xy}
\end{displaymath}
where $ab=1$.  The number of cycles is $\tfrac{1}{2} \phi(N)$.

The results are summarized in the following proposition:

\begin{proposition}
Let $N>1$ be an integer.  Then $X^+(N)$ contains no elliptic points.  The number of real cusps on $X^+(N)$ is
given by
\begin{displaymath}
\phi(N) \times
\begin{cases}
1 & \textrm{if $N$ is odd} \\
3 & \textrm{if $2 \parallel N$ but $N \ne 2$} \\
2 & \textrm{if $4 \mid N$ or $N=2$.}
\end{cases}
\end{displaymath}
The number of real components of $X^+(N)$ is given by
\begin{displaymath}
\begin{cases}
\psi(N) & \textrm{if $N$ is odd} \\
\tfrac{1}{2} \phi(N) & \textrm{if $N \ne 2$ is even} \\
1 & \textrm{if $N=2$.}
\end{cases}
\end{displaymath}
In all cases, each component has the same structure.
\end{proposition}

\subsection{The curve $X^-(N)$}
\label{ss:ex2}

Let $c_1$ denote the complex conjugation of $\mf{h}$ given by $z \mapsto 1/\ol{z}$.  Let $\Gamma=\Gamma^-(N)$ denote
the real congruence group $(\Gamma(N), c_1)$ and let $X=X^-(N)$ denote the
corresponding quotient.  It is not difficult to see that every real form of $\Gamma(N)$ is
equivalent to either $\Gamma^+(N)$ or $\Gamma^-(N)$, and that these two are equivalent to each other if and only if $N$
is odd.  The space $X^-(N)$ parameterizes pairs $(E, i)$ where $E$ is an elliptic curve and $i$ is an isomorphism of
$E[N]$ with the extension of $\Z/N\Z$ by $\mu_N$ corresponding under K\"ummer theory to $-1$ (in
$\R^{\times}/(\R^{\times})^N$ or even $\Q^{\times}/(\Q^{\times})^N$) under which the Weil
pairing corresponds to the standard pairing.  (The same description applies to $X^+(N)$, except one uses the
extension corresponding to $1 \in \R^{\times}/(\R^{\times})^N$.)  Of course, $\Gamma^-(N)$ has no elliptic points for
$N \ne 1$.  We assume $N>1$ in what follows.

When $r=0$ the real locus of $X^-(N)$ is isomorphic to that of $X^+(N)$.  Thus assume $r>0$.
The real cusps are represented by vectors of the form $(a, a)$ and $(a, -a)$ where $a$ is
prime to $N$.  When $N=2$ there is thus a single real cusp, and it has a self-edge of weight 2.  Now assume $N \ne 2$.
There are thus $\phi(N)$ real cusps.  Every cycle has two vertices and is of the form
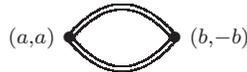
\begin{displaymath}
\begin{xy}
(-12, 0)*{\ss (a, a)}; (13, 0)*{\ss (b, -b)};
(-7, 0)*{\bullet}; (7, 0)*{\bullet};
{\ar@/^2.5ex/@{=}; (-7, 0); (7, 0)};
{\ar@/_2.5ex/@{=}; (-7, 0); (7, 0)};
\end{xy}
\end{displaymath}
where $ab=1$.  There are thus $\tfrac{1}{2} \phi(N)$ cycles.

The results are summarized in the following proposition:

\begin{proposition}
Let $N>1$ be an integer.  Then $X^-(N)$ contains no elliptic points.  The number of real cusps on $X^-(N)$ is given
by $\phi(N)$.  The number of real components of $X^-(N)$ is given by
\begin{displaymath}
\begin{cases}
\psi(N) & \textrm{if $N$ is odd} \\
\tfrac{1}{2} \phi(N) & \textrm{if $N \ne 2$ is even} \\
1 & \textrm{if $N=2$.}
\end{cases}
\end{displaymath}
In all cases, each real components has the same structure.
\end{proposition}

Note that $X^+(N)$ and $X^-(N)$ have the same number of real components, even though their real components have
different structure.

\subsection{The curve $X_1(N)$}
\label{ss:ex3}

Let $\Gamma=\Gamma_1(N)$ denote the real congruence group $(\Gamma_1(N), c_0)$ and let $X=X_1(N)$ be the corresponding
quotient.  The group $G$ is the subgroup of $\SL_2(\Z/N\Z)$ consisting of matrices of the form
\begin{displaymath}
\mat{\pm 1}{\ast}{}{\pm 1}.
\end{displaymath}
The space $X$ parameterizes pairs $(E, P)$ where $E$ is an elliptic curve and $P$ is a point on $E$ of
exact order $N$.  The group $\Gamma$ has no even order elliptic elements for $N>2$.

Suppose first that $r=0$.  The real cusps of $\Gamma$ are represented by vectors of the form $(a, 0)$ and $(0, a)$,
with $a$ prime to $N$.  This representation is unique, except for the action of $\pm 1$.  It follows that there are
$\phi(N)$ real cusps.  Clearly, $(a, 0)$ is connected in $\Xi_G$ to both $(0, b)$ and $(0, 2b)$, where $b=a^{-1}$.
The description of $\Xi$ is thus exactly the same as it was for $\Gamma(N)$ (either form).

Now suppose that $r=1$.  When $N=2$ there are two real cusps and one real elliptic point of even order in a single
component.  The picture is:
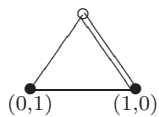
\begin{displaymath}
\begin{xy}
(-7, -7)*{\ss (0, 1)}; (7, -7)*{\ss (1,0)};
(-7, -5)*{\bullet}; (7, -5)*{\bullet}; (0, 5)*{\circ};
{\ar@{-}; (-7, -5); (7, -5)};
{\ar@{=}; (7, -5); (0, 5)};
{\ar@{-}; (-7, -5); (0, 5)};
\end{xy}
\end{displaymath}
Now suppose $N \ne 2$, so that there are no elliptic vertices.  The real cusps are represented by
vectors of the form $(a, 0)$, $(0, a)$, $(1, 2a)$ and $(a, t)$ with $a$ prime to $N$.
Replacing $a$ by $-a$ in each of the four forms yields an equivalent parabolic vertex, and this operation captures
all equivalences.  There are thus $\tfrac{1}{2} \phi(N)$ cusps of each kind, for a total of $2 \phi(N)$ real cusps.
Let $\epsilon=2+t$, so $\epsilon=1$ modulo 2 and $\epsilon=2$ modulo $N'$.  The local picture at $(a, 0)$ is then
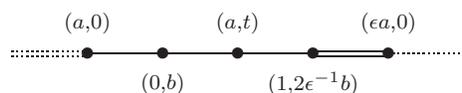
\begin{displaymath}
\begin{xy}
(-20, 4)*{\ss (a, 0)}; (-10, -4)*{\ss (0, b)}; (0, 4)*{\ss (a,t)};
(10, -4)*{\ss (1,2\epsilon^{-1} b)}; (20, 4)*{\ss (\epsilon a, 0)};
(-20, 0)*{\bullet}; (-10, 0)*{\bullet}; (0, 0)*{\bullet}; (10, 0)*{\bullet}; (20, 0)*{\bullet};
{\ar@{::}; (-30, 0); (-20, 0)};
{\ar@{-}; (-20, 0); (-10, 0)};
{\ar@{-}; (-10, 0); (0, 0)};
{\ar@{-}; (0, 0); (10, 0)};
{\ar@{=}; (10, 0); (20, 0)};
{\ar@{..}; (20, 0); (30, 0)};
\end{xy}
\end{displaymath}
Here $ab=1$.  We thus see that $(a, 0)$ is connected by a chain of length 4, involving each type of cusp exactly once,
to $(\epsilon a, 0)$.  It follows that the number of cycles is the cardinality of $(\Z/N\Z)^{\times}/\langle -1,
\epsilon \rangle$, which coincides with $\psi(N')$.

Finally, suppose that $r \ge 2$.  When $N=4$, there are three real cusps and one cycle, as follows:
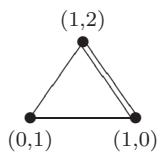
\begin{displaymath}
\begin{xy}
(-7, -8)*{\ss (0,1)}; (7, -8)*{\ss (1,0)}; (0, 8)*{\ss (1,2)};
(-7, -5)*{\bullet}; (7, -5)*{\bullet}; (0, 5)*{\bullet};
{\ar@{-}; (-7, -5); (7, -5)};
{\ar@{=}; (7, -5); (0, 5)};
{\ar@{-}; (-7, -5); (0, 5)};
\end{xy}
\end{displaymath}
Now suppose $N \ne 4$.  The real cusps of $\Gamma$ are represented by vectors of the form $(a, 0)$, $(0, a)$, $(1, 2a)$
and $(a, t)$ with $a$ prime to $N$.  There are $\tfrac{1}{2} \phi(N)$ cusps of each of the
first and second forms and $\tfrac{1}{4} \phi(N)$ cusps of each of the third and forth forms, for a total of
$\tfrac{3}{2} \phi(N)$ real cusps.  Every cycle of $\Xi$ contains exactly six cusps, and is of the form
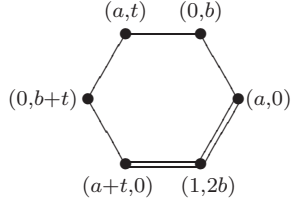
\begin{displaymath}
\begin{xy}
(10, 0)*{\bullet}; (5, 8.66)*{\bullet}; (-5, 8.66)*{\bullet};
(-10, 0)*{\bullet}; (-5, -8.66)*{\bullet}; (5, -8.66)*{\bullet};
(14, 0)*{\ss (a, 0)}; (5, 11.66)*{\ss (0, b)}; (-5, 11.66)*{\ss (a, t)};
(-16, 0)*{\ss (0, b+t)}; (-6, -11.66)*{\ss (a+t, 0)}; (6, -11.66)*{\ss (1, 2b)};
{\ar@{-}; (10, 0); (5, 8.66)};
{\ar@{-}; (5, 8.66); (-5, 8.66)};
{\ar@{-}; (-5, 8.66); (-10, 0)};
{\ar@{-}; (-10, 0); (-5, -8.66)};
{\ar@{=}; (-5, -8.66); (5, -8.66)};
{\ar@{=}; (5, -8.66); (10, 0)};
\end{xy}
\end{displaymath}
where $ab=1$.  It follows that there are $\tfrac{1}{4} \phi(N)$ cycles.

The above results are summarized in the following proposition:

\begin{proposition}
Let $N \ge 1$ be an integer, let $\Gamma=\Gamma_1(N)$ and let $c=c_0$.  Write $N=2^rN'$ with $N'$ odd.  Then
$X_{\Gamma}$ contains no even order elliptic points if $N>2$, and exactly one even order real elliptic point if $N$ is
1 or 2.  The number of real cusps on $X_{\Gamma}$ is given by
\begin{displaymath}
\phi(N) \times \begin{cases}
1 & \textrm{if $r=0$} \\
2 & \textrm{if $r=1$} \\
\tfrac{3}{2} & \textrm{if $r \ge 2$.}
\end{cases}
\end{displaymath}
The number of real components of $X_{\Gamma}$ is given by
\begin{displaymath}
\begin{cases}
\psi(N') & \textrm{if $r=0$ or $r=1$} \\
\tfrac{1}{4} \phi(N) & \textrm{if $r \ge 2$ and $N \ne 4$} \\
1 & \textrm{if $N=4$.}
\end{cases}
\end{displaymath}
In all cases, each component has the same structure.
\end{proposition}

\subsection{The curve $X_0(N)$}
\label{ss:ex4}

Let $\Gamma=\Gamma_0(N)$ denote the real congruence group $(\Gamma_0(N), c_0)$ and let $X=X_0(N)$ be the corresponding
quotient.  The group $G$ is the subgroup of $\SL_2(\Z/N\Z)$ consisting of upper triangular matrices.  The space $X$
parameterizes elliptic curves together with a cyclic subgroup of order $N$.  The set of cusps of $\Gamma$ is identified
with $\P^1(\Z/N\Z)$.  The action of $C$ on $\P^1(\Z/N\Z)$ is multiplication by $-1$.  The space $X$ can have
elliptic points, but as shown in Proposition~\ref{realell}, for $N>2$ none of them are real.

Write $G_N$ for the group $G$ discussed above and $\Xi_N$ for the corresponding graph.  If $N=p_1^{e_1} \cdots
p_n^{e_n}$ then $G_N=G_{p_1^{e_1}} \times \cdots \times G_{p_n^{e_n}}$, and each $G_{p_i^{e_i}}$ contains $-1$.
Thus $\Xi_N=\Xi_{p_1^{e_1}} \gtimes \cdots \gtimes \Xi_{p_n^{e_n}}$, where $\gtimes$ is the product discussed in
\S \ref{ss:modgr}.  The graphs $\Xi_{p^n}$ are as follows:
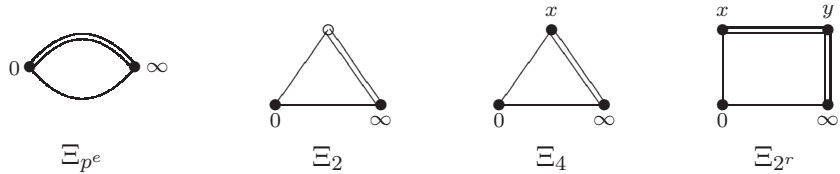
\begin{displaymath}
\begin{xy}
(0, -12)*{\Xi_{p^e}};
(-9, 0)*{\ss 0}; (10, 0)*{\ss \infty};
(-7, 0)*{\bullet}; (7, 0)*{\bullet};
{\ar@/^2.5ex/@{=}; (-7, 0); (7, 0)};
{\ar@/_2.5ex/@{-}; (-7, 0); (7, 0)};
\end{xy}
\hskip 8ex
\begin{xy}
(0, -12)*{\Xi_2};
(-7, -7)*{\ss 0}; (7, -7)*{\ss \infty};
(-7, -5)*{\bullet}; (7, -5)*{\bullet}; (0, 5)*{\circ};
{\ar@{-}; (-7, -5); (7, -5)};
{\ar@{=}; (7, -5); (0, 5)};
{\ar@{-}; (-7, -5); (0, 5)};
\end{xy}
\hskip 8ex
\begin{xy}
(0, -12)*{\Xi_4};
(-7, -7)*{\ss 0}; (7, -7)*{\ss \infty}; (0, 7.5)*{\ss x};
(-7, -5)*{\bullet}; (7, -5)*{\bullet}; (0, 5)*{\bullet};
{\ar@{-}; (-7, -5); (7, -5)};
{\ar@{=}; (7, -5); (0, 5)};
{\ar@{-}; (-7, -5); (0, 5)};
\end{xy}
\hskip 8ex
\begin{xy}
(0, -12)*{\Xi_{2^r}};
(-7, -7)*{\ss 0}; (7, -7)*{\ss \infty}; (-7, 7.5)*{\ss x}; (7, 7.5)*{\ss y};
(-7, -5)*{\bullet}; (7, -5)*{\bullet}; (-7, 5)*{\bullet}; (7, 5)*{\bullet};
{\ar@{-}; (-7, -5); (7, -5)};
{\ar@{=}; (7, -5); (7, 5)};
{\ar@{=}; (7, 5); (-7, 5)};
{\ar@{-}; (-7, 5); (-7, -5)};
\end{xy}
\end{displaymath}
Here $p$ is an odd prime, $e \ge 1$ and $r>2$ and $x=[1:2]$ and $y=[1:2^{r-1}]$.  Let $\Xi_{\rm odd}$ denote the
leftmost graph, so that any $\Xi_{p^e}$ with $p$ odd is isomorphic to $\Xi_{\rm odd}$.  Now,
$\Xi_{\rm odd} \ast \Xi_{\rm odd}$ is isomorphic to $\Xi_{\rm odd} \amalg \Xi_{\rm odd}$.  It follows that if
$N$ is an odd integer with $n$ distinct prime factors, then $\Xi_N$ is isomorphic to a disjoint union of
$2^{n-1}$ copies of $\Xi_{\rm odd}$.  In particular, $\Xi_N$ has $2^n$ vertices (all parabolic) and $2^{n-1}$
cycles.  Now, if $N=2^rN'$ with $r>0$ and $N'$ odd with $n$ distinct prime factors then $\Xi_N$ is isomorphic to a
disjoint union of $2^{n-1}$ copies of $\Xi_{2^r} \ast \Xi_{\rm odd}$.  These products are as follows:
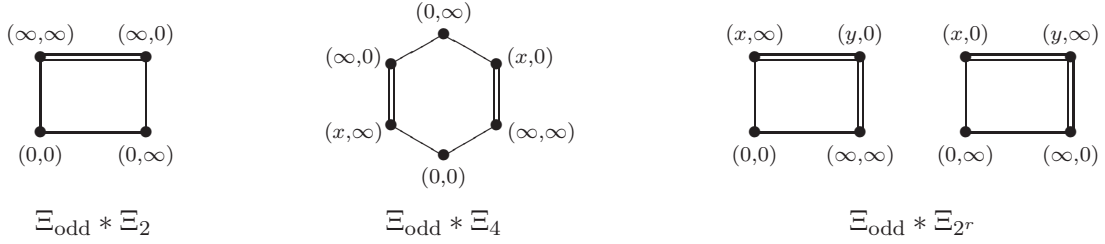
\begin{displaymath}
\begin{xy}
(-7, -8)*{\ss (0, 0)}; (7, -8)*{\ss (0, \infty)}; (-7, 8)*{\ss (\infty, \infty)}; (7, 8)*{\ss (\infty, 0)};
(-7, -5)*{\bullet}; (7, -5)*{\bullet}; (-7, 5)*{\bullet}; (7, 5)*{\bullet};
{\ar@{-}; (-7, -5); (7, -5)};
{\ar@{-}; (7, -5); (7, 5)};
{\ar@{=}; (7, 5); (-7, 5)};
{\ar@{-}; (-7, 5); (-7, -5)};
(0, -17)*{\Xi_{\rm odd} \gtimes \Xi_2};
\end{xy}
\hskip 12ex
\begin{xy}
(7, 4.04)*{\bullet}; (0, 8.08)*{\bullet}; (-7, 4.04)*{\bullet};
(-7, -4.04)*{\bullet}; (0, -8.08)*{\bullet}; (7, -4.04)*{\bullet};
(0, -11)*{\ss (0,0)}; (13, -5)*{\ss (\infty, \infty)}; (11.5, 5)*{\ss (x,0)};
(0, 11)*{\ss (0, \infty)}; (-12, 5)*{\ss (\infty, 0)}; (-12, -5)*{\ss (x, \infty)};
{\ar@{-}; (7, 4.04); (0, 8.08)};
{\ar@{-}; (0, 8.08); (-7, 4.04)};
{\ar@{=}; (-7, 4.04); (-7, -4.04)};
{\ar@{-}; (-7, -4.04); (0, -8.08)};
{\ar@{-}; (0, -8.08); (7, -4.04)};
{\ar@{=}; (7, -4.04); (7, 4.04)};
(0, -17)*{\Xi_{\rm odd} \gtimes \Xi_4};
\end{xy}
\hskip 12ex
\begin{xy}
(-21, 5)*{\bullet}; (-7, 5)*{\bullet}; (7, 5)*{\bullet}; (21, 5)*{\bullet};
(-21, -5)*{\bullet}; (-7, -5)*{\bullet}; (7, -5)*{\bullet}; (21, -5)*{\bullet};
(-21, -8)*{\ss (0, 0)}; (-7, -8)*{\ss (\infty, \infty)}; (-7, 8)*{\ss (y, 0)}; (-21, 8)*{\ss (x, \infty)};
(7, -8)*{\ss (0, \infty)}; (21, -8)*{\ss (\infty, 0)}; (21, 8)*{\ss (y, \infty)}; (7, 8)*{\ss (x, 0)};
{\ar@{-}; (-21, -5); (-7, -5)};
{\ar@{=}; (-7, -5); (-7, 5)};
{\ar@{=}; (-7, 5); (-21, 5)};
{\ar@{-}; (-21, 5); (-21, -5)};
{\ar@{-}; (21, -5); (7, -5)};
{\ar@{-}; (7, -5); (7, 5)};
{\ar@{=}; (7, 5); (21, 5)};
{\ar@{=}; (21, 5); (21, -5)};
(0, -17)*{\Xi_{\rm odd} \gtimes \Xi_{2^r}};
\end{xy}
\end{displaymath}
In particular, $\Xi_N$ has $2^{n+1}$, $3 \cdot 2^n$ or $2^{n+2}$ vertices, depending on if $r=1$, $r=2$ or $r \ge 3$
and $2^{n-1}$ or $2^n$ cycles, depending on if $r=1,2$ or $r\ge 3$.  These results are summarized in the following
proposition:

\begin{proposition}
Let $N \ge 1$ be an integer, write $N=2^rN'$, with $N'$ odd, and let $n$ be the number of distinct prime factors of
$N'$.  The space $X_0(N)$ has a real elliptic point only if $N$ is 1 or 2, in which case it has exactly 1.  The number
of real cusps on $X_0(N)$ is given by
\begin{displaymath}
2^n \times \begin{cases}
1 & \textrm{if $r=0$} \\
2 & \textrm{if $r=1$} \\
3 & \textrm{if $r=2$} \\
4 & \textrm{if $r \ge 3$}
\end{cases}
\end{displaymath}
The number of real components of $X_0(N)$ is 1 if $N$ is a power of 2 and is given by
\begin{displaymath}
2^{n-1} \times \begin{cases}
1 & \textrm{if $r \le 2$} \\
2 & \textrm{if $r \ge 3$}
\end{cases}
\end{displaymath}
otherwise.  In all cases, each component has the same structure.
\end{proposition}

\subsection{The curve $X_{\spl}(N)$}
\label{ss:ex5}

Let $G_N$ be the normalizer of the diagonal subgroup of $\SL_2(\Z/N\Z)$.  If $N$ is a prime power then $G_N$
consists of matrices of the form
\begin{displaymath}
\mat{a}{}{}{a^{-1}}, \qquad \mat{}{a}{-a^{-1}}{},
\end{displaymath}
with $a$ a unit of $\Z/p^e \Z$.  Let $\Gamma_{\spl}(N)$ denote
the inverse image of $G_N$ in $\SL_2(\Z)$, as well as the real congruence group $(\Gamma_{\spl}(N), c_0)$, and let
$X_{\spl}(N)$ denote the corresponding quotient.  Let $\Xi_N$ denote the graph corresponding to $G_N$.  If
$N=\prod p_i^{e_i}$ then $G=\prod G_{p_i^{e_i}}$, and so the graph
$\Xi_N$ decomposes into a product as well.  Thus it suffices to analyze $\Xi_{p^e}$ for prime $p$.

First suppose that $p$ is an odd prime.  Let $v=(0, 1)$, and for a unit $a$ of $\Z/p^e \Z$ let $u_a=(1, a)$.  Then
every parabolic vertex of $\Xi$ is represented by either $v$ of $u_a$.  The basis vectors $u_a$ and $u_{-a}$
represent the same parabolic vertex, and there are no other equivalences.  We thus find that there are
$1+\tfrac{1}{2} \phi(p^e)$ parabolic vertices.  The graphs for $p^e=3$ and $p^e=5$ are as follows:
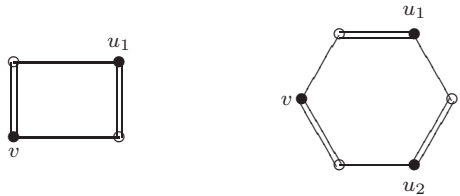
\begin{displaymath}
\begin{xy}
(-7, -7)*{\ss v}; (7, 7.5)*{\ss u_1};
(-7, -5)*{\bullet}; (7, -5)*{\circ}; (-7, 5)*{\circ}; (7, 5)*{\bullet};
{\ar@{-}; (-7, -5); (7, -5)};
{\ar@{=}; (7, -5); (7, 5)};
{\ar@{-}; (7, 5); (-7, 5)};
{\ar@{=}; (-7, 5); (-7, -5)};
\end{xy}
\hskip 12ex
\begin{xy}
(10, 0)*{\circ}; (5, 8.66)*{\bullet}; (-5, 8.66)*{\circ};
(-10, 0)*{\bullet}; (-5, -8.66)*{\circ}; (5, -8.66)*{\bullet};
(5, 11.66)*{\ss u_1}; (5, -11.66)*{\ss u_2}; (-12, 0)*{\ss v}; 
{\ar@{-}; (10, 0); (5, 8.66)};
{\ar@{=}; (5, 8.66); (-5, 8.66)};
{\ar@{-}; (-5, 8.66); (-10, 0)};
{\ar@{=}; (-10, 0); (-5, -8.66)};
{\ar@{-}; (-5, -8.66); (5, -8.66)};
{\ar@{=}; (5, -8.66); (10, 0)};
\end{xy}
\end{displaymath}
Suppose now that $p^e>5$.  Let $a$ be a unit of $\Z/p^e \Z$ and let $b$ be its inverse.  If
$a^2 \ne \pm 1, \pm \tfrac{1}{4}$ then the three vertices $u_a$, $u_b$ and $u_{b/4}$ are inequivalent, and the local
picture at $u_a$ is as follows:
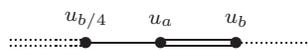
\begin{displaymath}
\begin{xy}
(-10, 0)*{\bullet}; (0, 0)*{\bullet}; (10, 0)*{\bullet};
(-10, 3)*{\ss u_{b/4}}; (0, 3)*{\ss u_a}; (10, 3)*{\ss u_b};
{\ar@{-}; (-10, 0); (0, 0)};
{\ar@{=}; (0, 0); (10, 0)};
{\ar@{::}; (-10, 0); (-20, 0)};
{\ar@{..}; (10, 0); (20, 0)};
\end{xy}
\end{displaymath}
The local picture near $u_1$ is as follows:
\begin{displaymath}
\begin{xy}
(-30, 0)*{\bullet}; (-20, 0)*{\bullet}; (-10, 0)*{\circ};
(0, 0)*{\bullet}; (10, 0)*{\circ}; (20, 0)*{\bullet}; (30, 0)*{\bullet};
(-30, 3)*{\ss u_{1/4}}; (-20, 3)*{\ss u_1}; (0, 3)*{\ss v}; (20, 3)*{\ss u_{1/2}}; (30, 3)*{\ss u_2};
{\ar@{-}; (-30, 0); (-20, 0)};
{\ar@{=}; (-20, 0); (-10, 0)};
{\ar@{-}; (-10, 0); (0, 0)};
{\ar@{=}; (0, 0); (10, 0)};
{\ar@{-}; (10, 0); (20, 0)};
{\ar@{=}; (20, 0); (30, 0)};
{\ar@{::}; (-30, 0); (-40, 0)};
{\ar@{..}; (30, 0); (40, 0)};
\end{xy}
\end{displaymath}
The above vertices are distinct unless $p^e$ if 7 or 9, in which case the two end vertices are identified.  Now
suppose that $a^2=-1$.  Such an element exists if and only if $p$ is 1 modulo 4.  The local picture is then:
\begin{displaymath}
\begin{xy}
(-20, 0)*{\bullet}; (-10, 0)*{\bullet}; (0, 0)*{\circ}; (10, 0)*{\bullet}; (20, 0)*{\bullet};
(-20, 3)*{\ss u_{a/4}}; (-10, 3)*{\ss u_a}; (10, 3)*{\ss u_{a/2}}; (20, 3)*{\ss u_{2a}};
{\ar@{-}; (-20, 0); (-10, 0)};
{\ar@{=}; (-10, 0); (0, 0)};
{\ar@{-}; (0, 0); (10, 0)};
{\ar@{=}; (10, 0); (20, 0)};
{\ar@{::}; (-20, 0); (-30, 0)};
{\ar@{..}; (20, 0); (30, 0)};
\end{xy}
\end{displaymath}
The above vertices are distinct (note $p^e$ cannot be 7 or 9 since $p$ is 1 modulo 4).
The total number of real elliptic points is 3 or 1 according to whether $p$ is 1 or 3 modulo 4.  In contrast to
previous cases, not all cycles of $\Xi$ have the same structure:  there are either one or two special cycles (those
containing the elliptic points).  If $p$ is 3 modulo 4, there is one cycle with elliptic points.  If $p$ is 1 modulo 4
then all three elliptic points belong to the same cycle if $\sqrt{-1}$ is a power of 2 (or, equivalently, if the
multiplicative order of 2 is divisible by 4), and otherwise belong to two distinct cycles.  Let $\sim$ be the
equivalence relation on $(\Z/p^e \Z)^{\times}$ given by $x \sim y$ if $x=\pm 4^n y^{\pm 1}$.  Let $\epsilon$ be 1
if the multiplicative order or 2 is divisible by 4.  Then the number of cycles of $\Xi$ is the cardinality of
$(\Z/p^e \Z)^{\times}/\sim$ minus $\epsilon$.  (This formula is valid even when $p^e$ is 3 or 5.)

Now suppose that $p=2$, and let $2^r$ be the prime power in question.  The graph for $r=1$ is as follows:
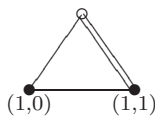
\begin{displaymath}
\begin{xy}
(-7, -7)*{\ss (1,0)}; (7, -7)*{\ss (1,1)};
(-7, -5)*{\bullet}; (7, -5)*{\bullet}; (0, 5)*{\circ};
{\ar@{-}; (-7, -5); (7, -5)};
{\ar@{=}; (7, -5); (0, 5)};
{\ar@{-}; (-7, -5); (0, 5)};
\end{xy}
\end{displaymath}
Now suppose that $r>1$.  Let $v=(1,0)$ and $v'=(1, t)$.  For a unit $a$ of $\Z/2^r\Z$, let $u_a=(1, a)$.  Then
every parabolic vertex is represented by $v$, $v'$ or some $u_a$.  The basis vectors $u_a$ and $u_{-a}$ are
equivalent.  There are no other equivalences.  We thus find that there are $2+\tfrac{1}{2} \phi(2^r)$ parabolic
vertices.  Let $a$ be unit of $\Z/2^r \Z$. and let $b$ be its inverse  If $a$ is not $\pm 1$ or $\pm 1+t$ then $u_a$
and $u_b$ are inequivalent and connected by two weight two edges:
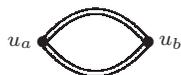
\begin{displaymath}
\begin{xy}
(-10, 0)*{\ss u_a}; (10, 0)*{\ss u_b};
(-7, 0)*{\bullet}; (7, 0)*{\bullet};
{\ar@/^2.5ex/@{=}; (-7, 0); (7, 0)};
{\ar@/_2.5ex/@{=}; (-7, 0); (7, 0)};
\end{xy}
\end{displaymath}
If $r>1$ then the vertex $u_{1+t}$ belongs to a loop of weight 2.  (When $r=1$, the vertex $u_{1+t}$ is equivalent
to $u_1$.)  The remaining vertices form a 5-cycle:
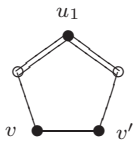
\begin{displaymath}
\begin{xy}
(6.66, 2.16)*{\circ}; (0, 7)*{\bullet}; (-6.66, 2.16)*{\circ}; (-4.11, -5.66)*{\bullet}; (4.11, -5.66)*{\bullet};
(0, 10)*{\ss u_1};  (-7.61, -5.66)*{\ss v}; (7.61, -5.66)*{\ss v'};
{\ar@{=}; (6.66, 2.12); (0, 7)};
{\ar@{=}; (0, 7); (-6.66, 2.12)};
{\ar@{-}; (-6.66, 2.12); (-4.11, -5.66)};
{\ar@{-}; (-4.11, -5.66); (4.11, -5.66)};
{\ar@{-}; (4.11, -5.66); (6.66, 2.12)};
\end{xy}
\end{displaymath}
We thus find that there are 2 elliptic vertices.  If $r=2$ there is a single cycle, the above 5-cycle.  If $r>2$ then
there are $1+\tfrac{1}{4} \phi(2^r)$ cycles.  As was the case for odd $p$, not all cycles have the same structure, but
there are at most two special cycles.

To determine $\Xi_N$ we must now take the product of the $\Xi_{p^{e_i}}$.  This is somewhat involved, and we have
not worked out the answer theoretically.  We have, however, written a program that computes these graphs;
it is available at \cite{Snowden}.  We mention two computational findings here.  First, up to $N=4000$, no real
component of $X_{\spl}(N)$ has more than 18 elliptic points of even order; the maximum 18 is attained several times,
first at $N=255$ where there is a cycle with 30 parabolic vertices and 18 elliptic vertices.  Does this bound hold
for all $N$?

Before mentioning the second computational finding, first note that for any real congruence group of odd level any
component has an even number of vertices.  We have seen several real components of even level with three vertices,
and $X_{\spl}(2^r)$ has a component with five vertices when $r>1$.  It is thus natural to wonder if one can have
real components containing other odd numbers of vertices.  Indeed, this is the case.  For instance, $X_{\spl}(10)$
has a single real component, with nine vertices, while $X_{\spl}(26)$ has a single real component, with 17 vertices.
For $N \le 4000$, the largest component with an odd number of vertices occurs on $X_{\spl}(3994)$ and has 2001
vertices.  It seems likely that with larger values of $N$ one will find arbitrarily large components with an odd
number of vertices.

\subsection{A Type 1b class}
\label{ss:ex6}

So far, we have only seen essentially one example of a loop:  it occurred in the graph of $\Gamma_{\spl}(2^r)$ for
$r>2$, and had weight two.  We now give a simpler example, which has weight one.  Let $g$ be the matrix
\begin{displaymath}
\mat{1}{1}{1}{}
\end{displaymath}
in $\SL_2(\Z/2\Z)$.  Then $g$ has order 3.  Let $\Gamma$ be the subgroup of $\SL_2(\Z)$ consisting of matrices
which reduce modulo 2 to 1, $g$ or $g^2$, and let $c=c_0$.  Then $(\Gamma, c)$ is a real congruence group.  The
only element of $\Gamma$ of finite even order is $-1$.  Indeed, if $\gamma$ is an element of even order then
its reduction modulo 2 is the identity, and it therefore belongs to $\Gamma(2)$; thus $\gamma=\pm 1$, as claimed.
Thus $X_{\Gamma}$ has no elliptic points.  The group generated by $g$ acts transitively on the non-zero vectors
in $(\Z/2\Z)^2$, and so $X_{\Gamma}$ has a single cusp, which is real.  It follows that $\Xi_{\Gamma}$ has a single
parabolic vertex connected to itself.  We thus see that the twisted conjugacy class of the identity element has
Type 1b.  This class always has weight one.

\subsection{A twisted form of $X_0(N)$}
\label{ss:ex7}

Let $\Gamma=\Gamma_0(N)$.  The matrix
\begin{displaymath}
\tau=\mat{}{-\sqrt{N}^{-1}}{\sqrt{N}}{}
\end{displaymath}
belongs to $\SL_2(\R)$, is admissible with respect to $c_0$ and normalizes $\Gamma$.  It follows that $c=c_0 \tau$
is a complex conjugation under which $\Gamma$ is stable.  Thus $(\Gamma, c)$ is a real Fuchsian group, and a twisted
form of $X_0(N)$; however, it is \emph{not} a real congruence group.  Explicitly, $c$ is given by $z \mapsto z
=(N \ol{z})^{-1}$.  A matrix
\begin{displaymath}
\mat{a}{b}{c}{d}
\end{displaymath}
is admissible with respect to $c$ if and only if $c=-Nb$.

\begin{proposition}
Let $\Gamma$ be as above, and let $\gamma \in \Gamma$ be admissible for $c$.  Then $C_{\gamma}$ contains a cusp
if and only if $N$ is a square.
\end{proposition}

\begin{proof}
Let
\begin{displaymath}
\gamma=\mat{a}{b}{-Nb}{d}.
\end{displaymath}
Then the limit points of $C_{\gamma}$ on $\R\P^1$ are
\begin{displaymath}
-\frac{b}{a} \pm \frac{1}{a \sqrt{N}}.
\end{displaymath}
These points are rational if and only if $N$ is a square.
\end{proof}

\begin{proposition}
If $N$ is divisible by 4 or a prime congruent to 3 mod 4 but is not a square then every admissible element of
$\Gamma$ is of Type 3.
\end{proposition}

\begin{proof}
The conditions guarantee that $\Gamma$ lacks even order elliptic points.
\end{proof}

\begin{proposition}
If $N=5$ then $\Gamma$ has two admissible twisted conjugacy classes, both of Type~4a.
\end{proposition}

\begin{proof}
The curve $X_{\Gamma}$ has exactly two elliptic points of order two, represented by $\tfrac{1}{5}(\pm 2+i)$.  Both
of these points have modulus $1/5$.  The curve $C_g$, with $g=1$, is exactly $\vert z \vert^2=1/5$.  Thus both elliptic
points are real.  As $X_{\Gamma}$ has genus 0 and a real point, it has one real component.  We have seen that there are
no real cusps.  This completes the proof.
\end{proof}

\begin{proposition}
If $N=2$ then $\Gamma$ has a single admissible twisted conjugacy class of Type~4b.
\end{proposition}

\begin{proof}
The curve $X_{\Gamma}$ has a single elliptic point of order two.  Since the elliptic point is unique, it must be real.
As $X_{\Gamma}$ has genus 0 and has a real point, it is isomorphic (over $\R$) to $\P^1$; thus $X_{\Gamma}$ has a
single real component.  This completes the proof.
\end{proof}

\newpage
\appendix
\section{Tables}

In this appendix we give tables of real components for certain modular curves.  In the tables, $g$ denotes the genus
of the curve in question, $\pi_0$ the number of real components, $p$ the number of real cusps and $e$ the number of
real elliptic points of even order.  In all but one case, there are no elliptic points when the level is greater
than two, and we omit $e$.

\mbox{}\vfill
\begin{center}
\small
\begin{tabular}{|rrrr||rrrr||rrrr||rrrr|}
\hline
$N$ & $g$ & $\pi_0$ & $p$ &   $N$ & $g$ & $\pi_0$ & $p$ &   $N$ & $g$ & $\pi_0$ & $p$ &   $N$ & $g$ & $\pi_0$ & $p$ \\
\hline
  1 &  0 & 1 &  1 &   31 &  2 & 1 &  2 &   61 &  4 & 1 &  2 &   91 &  7 & 2 &  4 \\
  2 &  0 & 1 &  2 &   32 &  1 & 1 &  4 &   62 &  7 & 1 &  4 &   92 & 10 & 1 &  6 \\
  3 &  0 & 1 &  2 &   33 &  3 & 2 &  4 &   63 &  5 & 2 &  4 &   93 &  9 & 2 &  4 \\
  4 &  0 & 1 &  3 &   34 &  3 & 1 &  4 &   64 &  3 & 1 &  4 &   94 & 11 & 1 &  4 \\
  5 &  0 & 1 &  2 &   35 &  3 & 2 &  4 &   65 &  5 & 2 &  4 &   95 &  9 & 2 &  4 \\
  6 &  0 & 1 &  4 &   36 &  1 & 1 &  6 &   66 &  9 & 2 &  8 &   96 &  9 & 2 &  8 \\
  7 &  0 & 1 &  2 &   37 &  2 & 1 &  2 &   67 &  5 & 1 &  2 &   97 &  7 & 1 &  2 \\
  8 &  0 & 1 &  4 &   38 &  4 & 1 &  4 &   68 &  7 & 1 &  6 &   98 &  7 & 1 &  4 \\
  9 &  0 & 1 &  2 &   39 &  3 & 2 &  4 &   69 &  7 & 2 &  4 &   99 &  9 & 2 &  4 \\
 10 &  0 & 1 &  4 &   40 &  3 & 2 &  8 &   70 &  9 & 2 &  8 &  100 &  7 & 1 &  6 \\
 11 &  1 & 1 &  2 &   41 &  3 & 1 &  2 &   71 &  6 & 1 &  2 &  101 &  8 & 1 &  2 \\
 12 &  0 & 1 &  6 &   42 &  5 & 2 &  8 &   72 &  5 & 2 &  8 &  102 & 15 & 2 &  8 \\
 13 &  0 & 1 &  2 &   43 &  3 & 1 &  2 &   73 &  5 & 1 &  2 &  103 &  8 & 1 &  2 \\
 14 &  1 & 1 &  4 &   44 &  4 & 1 &  6 &   74 &  8 & 1 &  4 &  104 & 11 & 2 &  8 \\
 15 &  1 & 2 &  4 &   45 &  3 & 2 &  4 &   75 &  5 & 2 &  4 &  105 & 13 & 4 &  8 \\
 16 &  0 & 1 &  4 &   46 &  5 & 1 &  4 &   76 &  8 & 1 &  6 &  106 & 12 & 1 &  4 \\
 17 &  1 & 1 &  2 &   47 &  4 & 1 &  2 &   77 &  7 & 2 &  4 &  107 &  9 & 1 &  2 \\
 18 &  0 & 1 &  4 &   48 &  3 & 2 &  8 &   78 & 11 & 2 &  8 &  108 & 10 & 1 &  6 \\
 19 &  1 & 1 &  2 &   49 &  1 & 1 &  2 &   79 &  6 & 1 &  2 &  109 &  8 & 1 &  2 \\
 20 &  1 & 1 &  6 &   50 &  2 & 1 &  4 &   80 &  7 & 2 &  8 &  110 & 15 & 2 &  8 \\
 21 &  1 & 2 &  4 &   51 &  5 & 2 &  4 &   81 &  4 & 1 &  2 &  111 & 11 & 2 &  4 \\
 22 &  2 & 1 &  4 &   52 &  5 & 1 &  6 &   82 &  9 & 1 &  4 &  112 & 11 & 2 &  8 \\
 23 &  2 & 1 &  2 &   53 &  4 & 1 &  2 &   83 &  7 & 1 &  2 &  113 &  9 & 1 &  2 \\
 24 &  1 & 2 &  8 &   54 &  4 & 1 &  4 &   84 & 11 & 2 & 12 &  114 & 17 & 2 &  8 \\
 25 &  0 & 1 &  2 &   55 &  5 & 2 &  4 &   85 &  7 & 2 &  4 &  115 & 11 & 2 &  4 \\
 26 &  2 & 1 &  4 &   56 &  5 & 2 &  8 &   86 & 10 & 1 &  4 &  116 & 13 & 1 &  6 \\
 27 &  1 & 1 &  2 &   57 &  5 & 2 &  4 &   87 &  9 & 2 &  4 &  117 & 11 & 2 &  4 \\
 28 &  2 & 1 &  6 &   58 &  6 & 1 &  4 &   88 &  9 & 2 &  8 &  118 & 14 & 1 &  4 \\
 29 &  2 & 1 &  2 &   59 &  5 & 1 &  2 &   89 &  7 & 1 &  2 &  119 & 11 & 2 &  4 \\
 30 &  3 & 2 &  8 &   60 &  7 & 2 & 12 &   90 & 11 & 2 &  8 &  120 & 17 & 4 & 16 \\
\hline
\end{tabular}
\normalsize\vskip.5\baselineskip
The curve $X_0(N)$.  See \S \ref{ss:ex4} for details.
\end{center}
\vfill

\newpage

\mbox{}\vfill

\begin{center}
\small
\begin{tabular}{|rrrrr||rrrrr||rrrrr||rrrrr|}
\hline
$N$ & $g$ & $\pi_0$ & $p_+$ & $p_-$ &   $N$ & $g$ & $\pi_0$ & $p_+$ & $p_-$ &   $N$ & $g$ & $\pi_0$ & $p_+$ & $p_-$ &   $N$ & $g$ & $\pi_0$ & $p_+$ & $p_-$ \\
\hline
 1 &  0 &  1 &  1 &  1 &   16 &   81 &  4 & 16 &  8 &   31 & 1001 &  3 & 30 & 30 &   46 & 2641 & 11 & 66 & 22 \\
 2 &  0 &  1 &  2 &  1 &   17 &  133 &  2 & 16 & 16 &   32 &  833 &  8 & 32 & 16 &   47 & 3773 &  1 & 46 & 46 \\
 3 &  0 &  1 &  2 &  2 &   18 &  109 &  3 & 18 &  6 &   33 & 1081 &  2 & 20 & 20 &   48 & 2689 &  8 & 32 & 16 \\
 4 &  0 &  1 &  4 &  2 &   19 &  196 &  1 & 18 & 18 &   34 & 1009 &  8 & 48 & 16 &   49 & 4215 &  1 & 42 & 42 \\
 5 &  0 &  1 &  4 &  4 &   20 &  169 &  4 & 16 &  8 &   35 & 1393 &  1 & 24 & 24 &   50 & 3301 & 10 & 60 & 20 \\
 6 &  1 &  1 &  6 &  2 &   21 &  241 &  1 & 12 & 12 &   36 & 1081 &  6 & 24 & 12 &   51 & 4321 &  2 & 32 & 32 \\
 7 &  3 &  1 &  6 &  6 &   22 &  241 &  5 & 30 & 10 &   37 & 1768 &  1 & 36 & 36 &   52 & 3865 & 12 & 48 & 24 \\
 8 &  5 &  2 &  8 &  4 &   23 &  375 &  1 & 22 & 22 &   38 & 1441 &  9 & 54 & 18 &   53 & 5500 &  1 & 52 & 52 \\
 9 & 10 &  1 &  6 &  6 &   24 &  289 &  4 & 16 &  8 &   39 & 1849 &  1 & 24 & 24 &   54 & 3889 &  9 & 54 & 18 \\
10 & 13 &  2 & 12 &  4 &   25 &  476 &  1 & 20 & 20 &   40 & 1633 &  8 & 32 & 16 &   55 & 5881 &  1 & 40 & 40 \\
11 & 26 &  1 & 10 & 10 &   26 &  421 &  6 & 36 & 12 &   41 & 2451 &  2 & 40 & 40 &   56 & 4801 & 12 & 48 & 24 \\
12 & 25 &  2 &  8 &  4 &   27 &  568 &  1 & 18 & 18 &   42 & 1729 &  6 & 36 & 12 &   57 & 6121 &  2 & 36 & 36 \\
13 & 50 &  1 & 12 & 12 &   28 &  529 &  6 & 24 & 12 &   43 & 2850 &  3 & 42 & 42 &   58 & 5461 & 14 & 84 & 28 \\
14 & 49 &  3 & 18 &  6 &   29 &  806 &  1 & 28 & 28 &   44 & 2281 & 10 & 40 & 20 &   59 & 7686 &  1 & 58 & 58 \\
15 & 73 &  1 &  8 &  8 &   30 &  577 &  4 & 24 &  8 &   45 & 2809 &  1 & 24 & 24 &   60 & 5185 &  8 & 32 & 16 \\
\hline
\end{tabular}
\normalsize\vskip.5\baselineskip
\parbox{5in}{The curves $X^{\pm}(N)$.  The two curves have the same genus and number of real components.  The column
$p_{\pm}$ indicates the number of real cusps on $X^{\pm}(N)$.  See \S \ref{ss:ex1} and \S \ref{ss:ex2} for details.}
\end{center}

\vskip2\baselineskip

\begin{center}
\small
\begin{tabular}{|rrrr||rrrr||rrrr||rrrr|}
\hline
$N$ & $g$ & $\pi_0$ & $p$ &   $N$ & $g$ & $\pi_0$ & $p$ &   $N$ & $g$ & $\pi_0$ & $p$ &   $N$ & $g$ & $\pi_0$ & $p$ \\
\hline
 1 &  0 & 1 &  1 &   26 & 10 & 1 & 24 &   51 &  65 & 2 & 32 &    76 & 136 &  9 & 54 \\
 2 &  0 & 1 &  2 &   27 & 13 & 1 & 18 &   52 &  55 & 6 & 36 &    77 & 181 &  1 & 60 \\
 3 &  0 & 1 &  2 &   28 & 10 & 3 & 18 &   53 &  92 & 1 & 52 &    78 & 121 &  1 & 48 \\
 4 &  0 & 1 &  3 &   29 & 22 & 1 & 28 &   54 &  52 & 1 & 36 &    79 & 222 &  1 & 78 \\
 5 &  0 & 1 &  4 &   30 &  9 & 1 & 16 &   55 &  81 & 1 & 40 &    80 & 137 &  8 & 48 \\
 6 &  0 & 1 &  4 &   31 & 26 & 3 & 30 &   56 &  61 & 6 & 36 &    81 & 190 &  1 & 54 \\
 7 &  0 & 1 &  6 &   32 & 17 & 4 & 24 &   57 &  85 & 2 & 36 &    82 & 171 &  2 & 80 \\
 8 &  0 & 1 &  6 &   33 & 21 & 2 & 20 &   58 &  78 & 1 & 56 &    83 & 247 &  1 & 82 \\
 9 &  0 & 1 &  6 &   34 & 21 & 2 & 32 &   59 & 117 & 1 & 58 &    84 & 133 &  6 & 36 \\
10 &  0 & 1 &  8 &   35 & 25 & 1 & 24 &   60 &  57 & 4 & 24 &    85 & 225 &  4 & 64 \\
11 &  1 & 1 & 10 &   36 & 17 & 3 & 18 &   61 & 126 & 1 & 60 &    86 & 190 &  3 & 84 \\
12 &  0 & 1 &  6 &   37 & 40 & 1 & 36 &   62 &  91 & 3 & 60 &    87 & 225 &  1 & 56 \\
13 &  2 & 1 & 12 &   38 & 28 & 1 & 36 &   63 &  97 & 3 & 36 &    88 & 181 & 10 & 60 \\
14 &  1 & 1 & 12 &   39 & 33 & 1 & 24 &   64 &  93 & 8 & 48 &    89 & 287 &  4 & 88 \\
15 &  1 & 1 &  8 &   40 & 25 & 4 & 24 &   65 & 121 & 4 & 48 &    90 & 153 &  1 & 48 \\
16 &  2 & 2 & 12 &   41 & 51 & 2 & 40 &   66 &  81 & 2 & 40 &    91 & 265 &  3 & 72 \\
17 &  5 & 2 & 16 &   42 & 25 & 1 & 24 &   67 & 155 & 1 & 66 &    92 & 210 & 11 & 66 \\
18 &  2 & 1 & 12 &   43 & 57 & 3 & 42 &   68 & 105 & 8 & 48 &    93 & 261 &  3 & 60 \\
19 &  7 & 1 & 18 &   44 & 36 & 5 & 30 &   69 & 133 & 1 & 44 &    94 & 231 &  1 & 92 \\
20 &  3 & 2 & 12 &   45 & 41 & 1 & 24 &   70 &  97 & 1 & 48 &    95 & 289 &  1 & 72 \\
21 &  5 & 1 & 12 &   46 & 45 & 1 & 44 &   71 & 176 & 1 & 70 &    96 & 193 &  8 & 48 \\
22 &  6 & 1 & 20 &   47 & 70 & 1 & 46 &   72 &  97 & 6 & 36 &    97 & 345 &  2 & 96 \\
23 & 12 & 1 & 22 &   48 & 37 & 4 & 24 &   73 & 187 & 4 & 72 &    98 & 235 &  1 & 84 \\
24 &  5 & 2 & 12 &   49 & 69 & 1 & 42 &   74 & 136 & 1 & 72 &    99 & 281 &  2 & 60 \\
25 & 12 & 1 & 20 &   50 & 48 & 1 & 40 &   75 & 145 & 1 & 40 &   100 & 231 & 10 & 60 \\
\hline
\end{tabular}
\normalsize\vskip.5\baselineskip
The curve $X_1(N)$.  See \S \ref{ss:ex3} for details.
\end{center}

\vfill

\newpage

\mbox{}\vfill

\begin{center}
\small
\begin{tabular}{|rrrrr||rrrrr||rrrrr||rrrrr|}
\hline
$N$ & $g$ & $\pi_0$ & $p$ & $e$ &  $N$ & $g$ & $\pi_0$ & $p$ & $e$ &  $N$ & $g$ & $\pi_0$ & $p$ & $e$ &  $N$ & $g$ & $\pi_0$ & $p$ & $e$ \\
\hline
 1 &  0 &  1 &  1 &  1 &   26 & 15 &  1 & 14 &  3 &   51 &  64 &  3 & 18 &  6 &   76 & 171 & 10 & 30 &  4 \\
 2 &  0 &  1 &  2 &  1 &   27 & 28 &  1 & 10 &  2 &   52 &  78 &  7 & 21 &  6 &   77 & 137 &  2 & 24 &  4 \\
 3 &  0 &  1 &  2 &  2 &   28 & 21 &  4 & 12 &  4 &   53 & 100 &  1 & 27 &  3 &   78 & 120 &  1 & 28 &  6 \\
 4 &  0 &  1 &  3 &  2 &   29 & 26 &  1 & 15 &  3 &   54 & 100 &  1 & 20 &  2 &   79 & 233 &  1 & 40 &  2 \\
 5 &  0 &  1 &  3 &  3 &   30 & 16 &  1 & 12 &  6 &   55 &  70 &  1 & 18 &  6 &   80 & 217 &  9 & 18 &  6 \\
 6 &  0 &  1 &  4 &  2 &   31 & 30 &  2 & 16 &  2 &   56 &  97 &  7 & 16 &  4 &   81 & 325 &  1 & 28 &  2 \\
 7 &  0 &  1 &  4 &  2 &   32 & 49 &  5 & 10 &  2 &   57 &  81 &  2 & 20 &  4 &   82 & 190 &  2 & 42 &  3 \\
 8 &  1 &  2 &  4 &  2 &   33 & 25 &  2 & 12 &  4 &   58 &  91 &  1 & 30 &  3 &   83 & 260 &  1 & 42 &  2 \\
 9 &  1 &  1 &  4 &  2 &   34 & 28 &  2 & 18 &  3 &   59 & 126 &  1 & 30 &  2 &   84 & 153 &  8 & 24 &  8 \\
10 &  1 &  1 &  6 &  3 &   35 & 27 &  1 & 12 &  6 &   60 &  79 &  6 & 18 & 12 &   85 & 172 &  4 & 27 &  9 \\
11 &  2 &  1 &  6 &  2 &   36 & 43 &  4 & 12 &  4 &   61 & 135 &  1 & 31 &  3 &   86 & 210 &  2 & 44 &  2 \\
12 &  3 &  2 &  6 &  4 &   37 & 45 &  1 & 19 &  3 &   62 & 105 &  2 & 32 &  2 &   87 & 196 &  1 & 30 &  6 \\
13 &  2 &  1 &  7 &  3 &   38 & 36 &  1 & 20 &  2 &   63 & 109 &  3 & 16 &  4 &   88 & 241 & 10 & 24 &  4 \\
14 &  3 &  1 &  8 &  2 &   39 & 36 &  1 & 14 &  6 &   64 & 225 &  9 & 18 &  2 &   89 & 301 &  3 & 45 &  3 \\
15 &  4 &  1 &  6 &  6 &   40 & 49 &  6 & 12 &  6 &   65 &  99 &  3 & 21 &  9 &   90 & 181 &  1 & 24 &  6 \\
16 &  9 &  3 &  6 &  2 &   41 & 57 &  2 & 21 &  3 &   66 &  85 &  2 & 24 &  4 &   91 & 192 &  3 & 28 &  6 \\
17 &  7 &  2 &  9 &  3 &   42 & 33 &  2 & 16 &  4 &   67 & 165 &  1 & 34 &  2 &   92 & 253 & 12 & 36 &  4 \\
18 &  7 &  1 &  8 &  2 &   43 & 63 &  2 & 22 &  2 &   68 & 136 &  9 & 27 &  6 &   93 & 225 &  4 & 32 &  4 \\
19 &  9 &  1 & 10 &  2 &   44 & 55 &  6 & 18 &  4 &   69 & 121 &  2 & 24 &  4 &   94 & 253 &  1 & 48 &  2 \\
20 & 10 &  3 &  9 &  6 &   45 & 55 &  1 & 12 &  6 &   70 &  91 &  1 & 24 &  6 &   95 & 216 &  1 & 30 &  6 \\
21 &  9 &  2 &  8 &  4 &   46 & 55 &  1 & 24 &  2 &   71 & 187 &  1 & 36 &  2 &   96 & 353 &  8 & 18 &  4 \\
22 & 10 &  1 & 12 &  2 &   47 & 77 &  1 & 24 &  2 &   72 & 193 &  7 & 16 &  4 &   97 & 360 &  2 & 49 &  3 \\
23 & 15 &  1 & 12 &  2 &   48 & 81 &  4 & 10 &  4 &   73 & 198 &  3 & 37 &  3 &   98 & 309 &  1 & 44 &  2 \\
24 & 17 &  2 &  6 &  4 &   49 & 94 &  1 & 22 &  2 &   74 & 153 &  1 & 38 &  3 &   99 & 271 &  2 & 24 &  4 \\
25 & 22 &  1 & 11 &  3 &   50 & 77 &  1 & 22 &  3 &   75 & 168 &  1 & 22 &  6 &  100 & 348 & 11 & 33 &  6 \\
\hline
\end{tabular}
\normalsize\vskip.5\baselineskip
The curve $X_{\spl}(N)$.  See \S \ref{ss:ex5} for details.
\end{center}

\vskip2\baselineskip

\noindent{\small
The genus of $X_{\spl}(N)$:  According to \cite[Prop.~1.40]{Shimura}, the genus of $X_{\spl}(N)$ is given by
\begin{displaymath}
g=1+\frac{\mu}{12}-\frac{\nu_2}{4}-\frac{\nu_3}{3}-\frac{\nu_{\infty}}{2},
\end{displaymath}
where $\mu$ is the index of $\Gamma_{\spl}(N)$ in $\SL_2(\Z)$ and $\nu_2$, $\nu_3$ and $\nu_{\infty}$ are the number
of equivalence classes of even order elliptic points, odd order elliptic points and cusps.  We do not know a reference
in the literature containing formulas for these quantities, so we give some here.  The index is given by
\begin{displaymath}
\mu=N^2 \prod_{p \mid N} \left( \frac{1+p^{-1}}{2} \right).
\end{displaymath}
The functions $\nu_2$, $\nu_3$ and $\nu_{\infty}$ are multiplicative functions of $N$.  On prime powers, we have
\begin{displaymath}
\nu_{\infty}(p^e)=\begin{cases}
p^{e-1} \left(\frac{p+1}{2} \right) & \textrm{if $p^e \ne 2$} \\
2 & \textrm{if $p^e=2$.}
\end{cases}
\qquad \textrm{and} \qquad
\nu_2(p^e)=\begin{cases}
p^{e-1} \left( \frac{p-1}{2} \right)+1 & \textrm{if $p=1 \pmod{4}$} \\
p^{e-1} \left( \frac{p+1}{2} \right) & \textrm{if $p=3 \pmod{4}$} \\
p^{e-1} & \textrm{if $p=2$.}
\end{cases}
\end{displaymath}
The value of $\nu_3(p^e)$ is 1 if $p$ is 1 modulo 3 and 0 otherwise (assuming $e>0$).
}

\vfill


\begin{thebibliography}{[GH]}

\bibitem[GH]{GrossHarris}
B.~H.~Gross and J.~Harris, \emph{Real algebraic curves}, Ann.\ Scient.\ \'Ecole Norm.\ Sup.\ (4) {\bf 14} (1981),
157--182.

\bibitem[Sh]{Shimura}
G.~Shimura, \emph{Introduction to the arithmetic theory of automorphic forms},  Pub.\ Math.\ Soc.\ Japan 11, I.~Shoten
and Princeton U.\ Press, 1971.

\bibitem[Sh2]{Shimura2}
G.~Shimura, \emph{On the real points of an arithmetic quotient of a bounded symmetric domain},
Math.\ Ann.\ {\bf 215} (1975), 135--164.

\bibitem[Sn]{Snowden}
A.~Snowden, \url{http://math.mit.edu/~asnowden/papers/xsplit.php}.

\bibitem[St]{Stein}
W.~Stein, \emph{Component Groups of $J_0(N)(\R)$ and $J_1(N)(\R)$},
\url{http://wstein.org/Tables/real_tamagawa/}, accessed on December 18, 2010.

\end{thebibliography}
\end{document}